\documentclass{amsart}
\usepackage{marktext}
\usepackage{gimac}

\newcommand{\GI}[2][]{\sidenote[colback=yellow!20]{\textbf{GI\xspace #1:} #2}}

\newcommand{\sfrac}[2]{#1/#2}

\newcommand{\Enu}{\mathcal E_\nu}
\newdelim{\ip}{\langle}{\rangle}

\newcommand{\SL}{\mathit{SL}}
\newcommand{\GL}{\mathit{GL}}

\makeatletter
\newcommand{\leqnomode}{\tagsleft@true\let\veqno\@@leqno}
\newcommand{\reqnomode}{\tagsleft@false\let\veqno\@@eqno}
\newcommand{\mylabel}[2]{\def\@currentlabel{#2}\label{#1}}
\makeatother

\DeclareMathOperator{\vol}{vol}

\theoremstyle{remark}
\newtheorem{question}[theorem]{Question}

\begin{document}
\title{Dissipation Enhancement by Mixing}
\author[Feng]{Yuanyuan Feng}
\address{%
  Department of Mathematical Sciences,
  Carnegie Mellon University,
  Pittsburgh, PA 15213.}
\email{yuanyuaf@andrew.cmu.edu}
\author[Iyer]{Gautam Iyer}
\address{%
  Department of Mathematical Sciences,
  Carnegie Mellon University,
  Pittsburgh, PA 15213.}
\email{gautam@math.cmu.edu}
\thanks{%
  This material is based upon work partially supported by
  the National Science Foundation under grants
  DMS-1252912, DMS-1814147
  and the Center for Nonlinear Analysis.
}
\keywords{Enhanced dissipation, mixing.}
\subjclass[2010]{Primary
  76F25; 
  Secondary
  37A25, 
  76R50. 
}

\begin{abstract}
  We quantitatively study the interaction between diffusion and mixing in both the continuous, and discrete time setting.
  In discrete time, we consider a mixing dynamical system interposed with diffusion.
  In continuous time, we consider the advection diffusion equation where the advecting vector field is assumed to be sufficiently mixing.
  The main results of this paper is to estimate the dissipation time and energy decay based on an assumption quantifying the mixing rate.
\end{abstract}
\maketitle

\section{Introduction.}\label{s:intro}

Diffusion and mixing are two fundamental phenomena that arise in a wide variety of applications ranging from micro-fluids to meteorology, and even cosmology.
In incompressible fluids, stirring induces mixing by filamentation and facilitates the formation of small scales.
Diffusion, on the other hand, efficiently damps small scales and the balance between these two phenomena is the main subject of our investigation.
Specifically, our aim in this paper is to quantify the interaction between diffusion and mixing in a manner that often arises in the context of fluids~\cites{DoeringThiffeault06,ConstantinKiselevEA08,LinThiffeaultEA11,Thiffeault12}.

In the absence of diffusion, the mixing of tracer particles passively advected by an incompressible flow has been extensively studied.
Several authors~\cites{MathewMezicEA05,LinThiffeaultEA11,Thiffeault12} measured mixing using \emph{multi-scale} norms and studied how efficiently incompressible flows can mix (see for instance~\cites{Bressan03,LunasinLinEA12,IyerKiselevEA14,AlbertiCrippaEA16,YaoZlatos17} and references therein).
In this scenario, however, there is no apriori limit to the resolution attainable via mixing.

In contrast, in the presence of diffusion, the effects of mixing may be enhanced, balanced, or even counteracted by diffusion (see for instance~\cites{FannjiangPapanicolaou94,ThiffeaultChildress03,FannjiangNonnenmacherEA04,ConstantinKiselevEA08,IyerNovikovEA10,KiselevXu15,MarcotteDoeringEA18,MilesDoering18}).
In this paper we quantify this interaction by studying the energy dissipation rate.
Roughly speaking, our main results can be stated as follows:
\begin{enumerate}
  \item
    In the continuous time setting we show (Theorem~\ref{t:disStrongCts}) that if the flow is strongly mixing, then the \emph{dissipation time} (i.e.\ the time required for the system to dissipate a constant fraction of its initial energy) can be bounded explicitly in terms of the mixing rate.
    In particular, for exponentially mixing flows, then the dissipation time is bounded by~$C\nu^{-\delta}$, where $\nu$ is the strength of the diffusion, and $\delta \in (0,1)$ is an explicit constant.
     If instead the flow is weakly mixing at a polynomial rate, then the dissipation time is bounded by~$C/(\nu|\ln \nu|^\delta)$ for some explicit $\delta >0$ (Theorem~\ref{t:disWeakCts}).

  \item 
    Under similar assumptions in the discrete time setting we obtain stronger bounds on the dissipation time (Theorems~\ref{t:disStrong} and~\ref{t:disWeak}).
    In particular, we show that the dissipation time of a pulsed diffusion with a map that is exponentially mixing is at most at most~$C\abs{\ln \nu}^2$.
    If the map is mixing at a polynomial rate, we show that the dissipation time is bounded by $C / \nu^\delta$ for some explicit $\delta \in (0, 1)$.
    

  \item
    In the discrete time setting we also show (Theorem~\ref{t:energydecay}) that the energy can not decay faster than double exponentially in time.
    Moreover,
    we obtain a family of examples where the energy indeed decays double exponentially in time.
    (In the continuous time setting the double exponential lower bound is known~\cite{Poon96}, however, to the best of our knowledge there are no smooth flows which are known to attain this lower bound.)

  \item
    In bounded domains, Berestycki et.\ al.~\cite{BerestyckiHamelEA05} studied asymptotics of the principal eigenvalue of the operator $-\nu \lap + u \cdot \grad$ as $\nu \to 0$.
    We show (Proposition~\ref{p:eval}) that one can use the dissipation time to obtain quantitative bounds on the rate at which the principal eigenvalue approaches~$0$.
\end{enumerate}

We remark that in the continuous time setting recent work of Coti Zelati \etal~\cite{CotiZelatiDelgadinoEA18} obtains a stronger bound on the dissipation time for two classes of strongly mixing flows.
Their result is discussed further below.

\subsection*{Plan of this paper.}

We begin by defining mixing rates, and state our main results in Section~\ref{s:results}.
Next, in Section~\ref{s:discrete}, we prove the dissipation time bounds in the discrete time setting (Theorems~\ref{t:disStrong} and~\ref{t:disWeak}).
In Section~\ref{s:toral} we study toral automorphisms, and use them to prove our result on energy decay (Theorem~\ref{t:energydecay}).
These proofs require certain facts on algebraic number fields, and may be skipped by readers who are not familiar with this material.
In Section~\ref{s:continuous} we prove the dissipation time bounds in the continuous time setting.
The proofs are similar to the discrete case, with a few key differences that we highlight.
Finally we conclude this paper with two appendices.
The first (Appendix~\ref{s:mixrates}) provides a brief introduction to mixing rates and the notions used to formulate our results.
The second (Appendix~\ref{s:ckrz}) shows that the characterization of relaxation enhancing flows in~\cites{ConstantinKiselevEA08,KiselevShterenbergEA08} still applies in the context of pulsed diffusions.

\subsection*{Acknowledgements.}

We would like to thank
  Giovani Alberti,
  Boris Bukh,
  Gianluca Crippa,
  Charles R. Doering,
  Tarek M. Elgindi,
  Albert Fannjiang,
  Anna L. Mazzucato,
  Jean-Luc Thiffeault,
  and
  Xiaoqian Xu
  for many helpful discussions.

\section{Main Results.}\label{s:results}

We devote this section to stating our main results.
In the discrete time setting we consider \emph{pulsed diffusions} (mixing maps interposed with diffusion), and our results concerning these are stated in Section~\ref{s:pulsed}, below.
In the continuous time setting we consider the advection diffusion equation, and our results in this setting are stated in Section~\ref{s:ad}, below.

\subsection{Pulsed Diffusions.}\label{s:pulsed}

In our setup we will consider a mixing map on a closed Riemannian manifold.
While the primary manifold we are interested in is the torus, there are, to the best of our knowledge, no known examples of smooth exponentially mixing maps on the torus that can be realized as the time one map of the flow of a smooth incompressible vector field.
There are, however, several examples of closed Riemannian manifolds that admit such maps (see~\cites{Dolgopyat98,ButterleyWar16} and references therein).
Since working on closed Riemannian manifolds does not increase the complexity by much, we state our results in this context instead of restricting our attention to the torus.

Let $M$ be a closed $d$-dimensional Riemannian manifold, and $\varphi \colon M \to M$ be a smooth volume preserving diffeomorphism.
For simplicity we will subsequently assume that the volume form on $M$ is normalized so that the total volume, $\abs{M}$, is~$1$.
Let $\nu > 0$ be the strength of the diffusion, $\lap$ denote the Laplace-Beltrami operator on~$M$, and $L^2_0 = L^2_0(M)$ denote the space of all mean zero square integrable functions on~$M$.
Given $\theta_0 \in L^2_0$, we consider the \emph{pulsed diffusion} defined by
\begin{align}\label{e:pulseddiff2}
  \theta_{n+1}=e^{\nu\Delta} U \theta_n \,.
\end{align}
Here $U \colon L^2(M) \to L^2(M)$ is the Koopman operator associated with~$\varphi$, and is defined by $Uf = f \circ \varphi$.
Our aim is to understand the asymptotic behaviour of the energy $\norm{\theta_n}_{L^2_0}$ in the long time, small diffusivity limit.
For notational convenience, we will use $\norm{\cdot}$ to denote the $L^2_0$ norm, and $\ip{\cdot, \cdot}$ to denote the $L^2_0$ inner-product.

Since $\varphi$ is volume preserving, the operator~$U$ is unitary and hence if $\nu = 0$ the system~\eqref{e:pulseddiff2} conserves energy.
If $\nu > 0$ and $\varphi$ is mixing, then Koopman operator~$U$ produces fine scales which are rapidly damped by the diffusion.
We quantify this using the notion of \emph{dissipation time} in~\cites{FannjiangWoowski03} (see also~\cites{FannjiangNonnenmacherEA04,FannjiangNonnenmacherEA06}).

\begin{definition}[Dissipation time]\label{d:dissipationTime}
  We define the \emph{dissipation time} of the operator $U$ by
  \begin{align*}
    \tau_d &\defeq
	\inf \set[\Big]{
	    n \in \N \st \norm{ (e^{\nu \lap} U) ^n }_{L^2_0 \to L^2_0} < \frac{1}{e}
	  }\\
	&= \inf \set[\Big]{
	    n \in \N \st \norm{\theta_n} < \frac{\norm{\theta_0}}{e}
	    \quad\text{for all } \theta_0 \in L^2_0
	  }
	\,.
  \end{align*}
\end{definition}

Since $U$ is unitary we clearly have $\norm{\theta_n} \leq e^{-\nu \lambda_1} \norm{ \theta_{n-1}}$, where $\lambda_1 > 0$ is the smallest non-zero eigenvalue of $-\lap$ on $M$.
Consequently, we always have
\begin{equation}\label{e:TdStupid}
  \tau_d \leq \frac{1}{\nu \lambda_1}\,,
\end{equation}
Our aim is to investigate how~\eqref{e:TdStupid} can be improved given an assumption on the mixing properties of $\varphi$.
In continuous time, Constantin et.\ al.~\cite{ConstantinKiselevEA08} (see also~\cite{KiselevShterenbergEA08}) characterized flows for which the dissipation time is $o(1/\nu)$.
Their result can directly be adapted to pulsed diffusions as follows.
\begin{proposition}\label{p:re}
  The Koopman operator~$U$ has no eigenfunctions in $\dot H^1$ if and only if
  \begin{equation*}
    \lim_{\nu \to 0} \nu \tau_d = 0\,.
  \end{equation*}
\end{proposition}
Since the proof is a direct adaptation of~\cites{ConstantinKiselevEA08,KiselevShterenbergEA08}, we relegate it to Appendix~\ref{s:ckrz}.
We remark, however, that without a quantitative assumption on the mixing rate of~$\varphi$, it does not seem possible to obtain more information regarding the rate at which $\nu \tau_d \to 0$.

Our main results obtain bounds for the rate at which $\nu \tau_d \to 0$ in terms of the mixing rate of~$\varphi$.
Recall, (strongly) mixing maps are those for which the correlations $\ip{U^n f, g}$ decay to~$0$ as $n \to \infty$ for all $f, g \in L^2_0$.
Weakly mixing maps are those for which the Ces\`aro averages of $\abs{\ip{U^n f, g}}^2$ decay to $0$ (see Appendix~\ref{s:mixrates} for a brief introduction and~\cite{EisnerFarkasEA15,KatokHasselblatt95,SturmanOttinoEA06} for a comprehensive treatment).
We quantify the mixing rate of $\varphi$ by imposing a rate at which these convergences occur.

\begin{definition}\label{d:mixrate}
  Let $h \colon [0,\infty) \to (0, \infty)$ be a decreasing function that vanishes at infinity.
  \begin{enumerate}\reqnomode
    \item
      Given $\alpha, \beta > 0$, we say that $\varphi$ is \emph{strongly $\alpha$, $\beta$ mixing with rate function~$h$} if for all $f \in \dot H^\alpha$, $g \in \dot H^\beta$ and $n \in \N$  the associated Koopman operator $U$ satisfies 
      \begin{equation}\label{e:hmixrateStrong}
	\abs[\big]{\ip{U^n f, g}}
	  \leq h\paren{n} \norm{f}_{\alpha}  \norm{g}_{\beta}\,.
      \end{equation}

    \item
      Given $\alpha, \beta \geq 0$, we say that $\varphi$ is \emph{weakly 
	$\alpha$, $\beta$ mixing with rate function~$h$} if for all $f \in \dot H^\alpha$, $g\in \dot H^\beta$ and $n \in \N$ the associated Koopman operator $U$ satisfies 
      \begin{equation}\label{e:hmixrateWeak}
	\paren[\Big]{ \frac{1}{n} \sum_{k = 0}^{n-1}
	  \abs[\big]{\ip{U^kf,g}}^2 }^{1/2}
	  \leq h\paren{n} \norm{f}_{\alpha}  \norm{g}_{\beta}\,.
      \end{equation}
  \end{enumerate}
\end{definition}

Here $\dot H^\alpha = \dot H^\alpha(M)$ is the homogeneous Sobolev space of order $\alpha$, and $\norm{\cdot}_\alpha$ denotes the norm in~$\dot H^\alpha$.
In the dynamical systems literature it is common to use H\"older spaces instead of Sobolev spaces, and study strongly mixing maps that are \emph{exponentially mixing} (i.e.\ $h(t) = c_1 e^{-c_2 t}$ for some $c_1 < \infty$ and $c_2 > 0$).
Using Sobolev spaces and asymmetric norms on $f$ and $g$, however, is more convenient for our purposes.
In order not to detract from our main results, we briefly motivate and study the above notions of mixing in Appendix~\ref{s:mixrates}.
Our main results on the dissipation time are as follows:

\begin{theorem}\label{t:disStrong}
  Let $\alpha, \beta > 0$, and $h \colon [0,\infty) \to (0, \infty)$ be a decreasing function that vanishes at infinity.
  If $\varphi$ is strongly $\alpha$, $\beta$ mixing with rate function~$h$, then
  the dissipation time is bounded by 
  \begin{align}\label{e:taud}
  \tau_d \leq \frac{C}{\nu H_1(\nu)}\,.
  \end{align}
  Here $C$ is a universal constant which can be chosen to be~$34$, and $H_1\colon (0, \infty) \to (0, \infty)$ is defined by
  \begin{equation}\label{e:H1}
    H_1(\mu) \defeq \sup \set[\Big]{
      \lambda \st
	h \paren[\Big]{ \frac{1}{2 \sqrt{\lambda \mu}}}
	\leq
	\frac{\lambda^{-(\alpha + \beta) / 2}}{2}
	} \,.
  \end{equation}
\end{theorem}

Before proceeding further, we compute the dissipation time~$\tau_d$ in two useful cases.

\begin{corollary}\label{c:strong}
  Let $\alpha, \beta, h, \varphi$ be as in Theorem~\ref{t:disStrong}.
  \begin{enumerate}\reqnomode
    \item
      If the mixing rate function $h\colon (0, \infty) \to (0, \infty)$ is the power law
      \begin{equation}\label{e:powerh}
	h(t) = \frac{c}{t^{p}}\,,
      \end{equation}
      for some $p>0$, then the dissipation time is bounded by
      \begin{equation}\label{e:taudstrongPower}
	\tau_d \leq \frac{C}{\nu^{\delta}}
	\qquad\text{where }
	\delta \defeq \frac{\alpha+\beta}{\alpha+\beta+p} \,,
      \end{equation}
      and~$C = C(c, \alpha, \beta, p) > 0$ is a finite constant

    \item
      If the mixing rate function $h\colon [0, \infty) \to (0, \infty)$ is the exponential function
      \begin{align}\label{e:exp}
	h(t) =c_1\exp(-c_2 t)\,,
      \end{align}
      for some constants $c_1, c_2 > 0$, then the dissipation time is bounded by
      \begin{equation}\label{e:taudstrongExp}
	\tau_d \leq C \abs{\ln \nu}^2\,,
      \end{equation}
      and~$C = C(c_1, c_2, \alpha, \beta) > 0$ is a finite constant
  \end{enumerate}
\end{corollary}
\begin{remark}\label{r:taudExpImproved}
  In the proof of Corollary~\ref{c:strong} (page~\pageref{pg:c:strong}) we will see that the bound~\eqref{e:taudstrongExp} can be improved to a bound of the form
  \begin{equation*}
    \tau_d \leq C_0 \paren[\Big]{
      \abs{\ln \nu} - C_1 \ln \abs[\big]{ \ln \nu - \ln \abs{\ln \nu}} 
    }^2
  \end{equation*}
  for explicit constants $C_0$, $C_1$  depending only on $c_1$, $c_2$, $\alpha$, $\beta$ and the constant~$C$ appearing in~\eqref{e:taud}.
  However, since~$C$ is not optimal, this improvement is not significant.
\end{remark}

When~$\varphi$ is weakly mixing, the bounds we obtain for the dissipation time are weaker than that in Theorem~\ref{t:disStrong}.
We state these results next.

\begin{theorem}\label{t:disWeak}
  Let $\alpha, \beta \geq 0$, and $h \colon [0,\infty) \to (0, \infty)$ be a decreasing function that vanishes at infinity.
  If $\varphi$ is weakly $\alpha$, $\beta$ mixing with rate function~$h$, then
  the dissipation time is bounded by 
  \begin{equation}\label{e:taudWeak}
    \tau_d \leq \frac{C}{\nu H_2(\nu)}\,.
  \end{equation}
  Here $C$ is a universal constant which can be chosen to be~$34$, and $H_2 \colon (0,\infty)\to (0,\infty)$ is defined by
  \begin{align}\label{e:H2}
    H_2(\mu) \defeq \sup\set[\Big]{
 \lambda \st   
      h\paren[\Big]{ \frac{1}{2\sqrt{\mu\lambda}  } }\leq \frac{1}{2\sqrt{\tilde c} } \lambda^{ -(2\alpha + 2\beta + d) / 4 }
  }\,,
  \end{align}
  where $\tilde c = \tilde c(M) > 0$ is a finite constant that only depends on the manifold $M$.
\end{theorem}

\begin{remark}\label{r:ctilde}
  We will see in the proof of Theorem~\ref{t:disWeak} that the constant~$\tilde c$ can be determined by the asymptotic growth of the eigenvalues of the Laplacian on~$M$.
  Explicitly, let $0 < \lambda_1 < \lambda_2 \leq\cdots$ be the eigenvalues of the Laplacian, where each eigenvalue is repeated according to its multiplicity.
  Then for any $\epsilon \in (0, 1)$ we can choose
  \begin{equation*}
    \tilde c
      = (1 + \epsilon) \lim_{j \to \infty} \frac{j}{\lambda_j^{d/2}}
      = \frac{(1 + \epsilon)\vol(M)}{ (4\pi)^{d/2}\, \Gamma(\frac{d}{2}+1)}\,.
  \end{equation*}
  The existence, and precise value, of the limit above is given by Weyl's lemma (see for instance~\cite{MinakshisundaramPleijel49}).
\end{remark}

We now compute~$\tau_d$ explicitly when the weak mixing rate function~$h$ decays polynomially. 

\begin{corollary}\label{c:taudWeakPower}
  Let $\alpha, \beta, h, \varphi$ be as in Theorem~\ref{t:disWeak}.
  If the mixing rate function $h$ is the power law~\eqref{e:powerh} for some $p \in (0, 1/2]$%
 \footnote{
    We require $p \in (0, 1/2]$, instead of $p > 0$, as the weak mixing rate can never be faster than~$1 / \sqrt{n}$.
    This can be seen immediately by choosing~$f = g$ in~\eqref{e:hmixrateWeak}.
  },
  then the dissipation time is bounded by
  \begin{equation}\label{e:taudWeakPower}
    \tau_d \leq C \nu^{-\delta}\,,
    \qquad\text{where}\quad
    \delta \defeq \frac{d+2\alpha+2\beta}{d+2p+2\alpha+2\beta} \,,
  \end{equation}
  and $C = C(\varphi, M, s,\alpha, \beta)$ is some finite constant.
\end{corollary}

\begin{remark}
  Note that as $\nu \to 0$, both $H_1(\nu) \to \infty$ and $H_2(\nu) \to \infty$.
  Thus the bounds obtained in both Theorems~\ref{t:disStrong} and~\ref{t:disWeak}, guarantee $\nu \tau_d \to 0$ as~$\nu \to 0$, and hence are stronger than the elementary bound~\eqref{e:TdStupid}.
\end{remark}

\begin{remark}
  Notice that if~$\varphi$ is strongly $\alpha$, $\beta$ mixing with rate function~$h$, then it is also weakly $\alpha$, $\beta$ mixing with rate function~$h_w$, where $h_w \colon [0, \infty) \to (0, \infty)$ is any continuous decreasing function such that
  \begin{equation*}
    h_w(n) \defeq
      \paren[\Big]{
	\frac{1}{n} \sum_{k = 0}^{n-1} h(k)^2
      }^{1/2}
      \qquad\text{for every}\quad n \in \N\,.
  \end{equation*}
  In this case, however, one immediately sees that the bound provided by Theorem~\ref{t:disWeak} is weaker than that provided by Theorem~\ref{t:disStrong}.
  In particular, suppose~$\varphi$ is strongly $\alpha$, $\beta$ mixing with rate function~$h$ given by the power law~\eqref{e:powerh} for some $p \in (0, 1/2]$.
  Then~$\varphi$ is also weakly $\alpha$, $\beta$ mixing with rate function given by
  \begin{equation*}
    h_w(t) = \begin{dcases}
      \frac{C_p}{ (1+t)^p} & p < 1/2\,,\\
      \paren[\Big]{\frac{C_p \ln (1 + t)}{1+t}}^{1/2} & p = 1/2\,,
    \end{dcases}
  \end{equation*}
  for some constant $C_p = C_p(c, p)$.
  In this case Corollary~\ref{c:taudWeakPower} applies when $p<1/2$, and asserts that the dissipation time~$\tau_d$ is bounded by~\eqref{e:taudWeakPower}.
  This, however, is weaker than~\eqref{e:taudstrongPower}.
\end{remark}

Before proceeding further, we note that Fannjiang et.\ al.~\cite{FannjiangNonnenmacherEA04} (see also~\cites{FannjiangWoowski03,FannjiangNonnenmacherEA06}) also obtain bounds on the dissipation time~$\tau_d$ assuming the time decay of the correlations of the \emph{diffusive operator}~$e^{\nu \lap} U$ for sufficiently small~$\nu$.
Explicitly they assume sufficient decay of~$\ip{(e^{\nu \lap} U)^n f, g}$ as $n \to \infty$, and then show that the dissipation time $\tau_d$ is at most $C / \abs{\ln \nu}$.
In contrast, our results only assume decay of the correlations of the operator~$U$ (without diffusion) as in Definition~\ref{d:mixrate}.

We now turn to studying the energy decay as $n \to \infty$.
Clearly
\begin{equation*}
  \norm{\theta_n}
    \leq \norm[\Big]{ \paren[\big]{ (e^{\nu \lap} U)^{\tau_d}}^{\floor{n / \tau_d} } \theta_0}
    \leq \norm[\big]{ (e^{\nu \lap} U)^{\tau_d}}^{\floor{n / \tau_d} } \norm{\theta_0}
    \leq e^{-\floor{n / \tau_d}}  \norm{\theta_0}\,,
\end{equation*}
and thus the energy $\norm{\theta_n}$ decays at least exponentially with rate $1 / \tau_d$ as $n \to \infty$.
This bound, however, is not optimal.
Indeed, if~$\varphi$ is the Arnold cat map, it is known~\cite{ThiffeaultChildress03} that the energy decays double exponentially.
We show that this remains true for a large class of toral automorphisms.
Moreover, Poon~\cite{Poon96} proved a matching lower bound for the continuous time advection diffusion equation.
This is readily adapted to the discrete time setting.

\begin{theorem}[Energy decay]\label{t:energydecay}
  For any $\theta_0 \in \dot H^1$, there exist finite constants $C > 0$ and $\gamma = \gamma(\norm{\varphi}_{C^1}) > 1$ for which the double exponential lower bound
  \begin{equation}\label{e:dexp}
    \norm{\theta_n}^2 \geq
      \norm{\theta_0}^2 \exp\paren[\Big]{
	- \frac{C \nu \norm{\theta_0}_1^2}{\norm{\theta_0}^2}
	\gamma^n
      }\,,
  \end{equation}
  holds.
  Moreover, there exists a smooth, volume preserving diffeomorphism on the torus for which the above bound is achieved.
  Explicitly, if~$\varphi$ is any toral automorphism which has no proper invariant rational subspaces, and has no eigenvalues that are roots of unity, then there exists finite constants $C$ and $\gamma > 1$ such that
  \begin{equation}\label{e:dexpUpper}
    \norm{\theta_n}^2 \leq \norm{\theta_0}^2 \exp\paren[\Big]{ - \frac{\nu \gamma^n}{C} }\,,
  \end{equation}
  for all $\theta_0 \in L^2_0$.
\end{theorem}

\begin{remark}
  Note, even though both~\eqref{e:dexp} and~\eqref{e:dexpUpper} are double exponential in time, the decay rates do not match.
  Namely, the constant in the first exponential in~\eqref{e:dexp} depends on the initial data and is large for ``highly mixed'' initial data.
  On the other hand, the exponential factor in~\eqref{e:dexpUpper} is universal, and independent of the initial data.
\end{remark}

We prove Theorem~\ref{t:energydecay} in Section~\ref{s:toral}.
Recall \emph{toral automorphisms} are diffeomorphisms of the torus onto itself that can be lifted to a linear transformation on the covering space~$\R^d$, and Section~\ref{s:toral} also contains a brief introduction to such maps.

\begin{remark}\label{r:taudLog}
  The lower bound~\eqref{e:dexp} immediately implies that the dissipation time can always be bounded \emph{below} by
  \begin{equation}\label{e:taudLogLower}
    \tau_d \geq C \abs{\ln \nu}\,,
  \end{equation}
  for some constant $C = C( \norm{\varphi}_{C^1} )$.
For maps~$\varphi$ that achieve the upper bound~\eqref{e:dexpUpper}, the dissipation time also satisfies the matching upper bound
\begin{equation}\label{e:taudExp}
  \tau_d \leq C \abs{\ln \nu}\,.
\end{equation}
  In the best case scenario, our results (Theorem~\ref{t:disStrong} and Corollary~\ref{c:strong}) show that for exponentially mixing maps we have $\tau_d \leq C \abs{\ln \nu}^2$, missing this bound by a factor of $\abs{\ln \nu}$.
  While we produce (Proposition~\ref{p:catmapExpMix}, below)
  a family of exponentially mixing diffeomorphisms for which the dissipation time is of order~$\abs{\ln \nu}$, we do not know if this is true for general exponentially mixing diffeomorphisms.
\end{remark}

\subsection{Advection Diffusion Equation.}\label{s:ad}

We now turn to the continuous time setting.
Let $M$ be a (smooth) closed Riemannian manifold, and $u$ be a smooth, time dependent, divergence free vector field on~$M$.
Let $\theta$ be a solution to the advection-diffusion equation
\begin{equation}\label{e:ad}
  \left\{
  \begin{alignedat}{2}
    \span
      \partial_t \theta_s + (u(t) \cdot \grad) \theta_s - \nu \lap \theta_s = 0
      &\qquad& \text{in $M$, for $t > s$,}\\
    \span
      \theta_s(t) = \theta_{s,0} && \text{for $t = s$.}
  \end{alignedat}\right.
\end{equation}
for $t > s$, with initial data~$\theta_s(s) = \theta_{s,0} \in L^2_0(M)$.
Since $u$ is divergence free we have
\begin{equation}\label{e:energyCts}
  \frac{1}{2} \partial_t \norm{\theta_s(t)}^2 + \nu \norm{\theta_s(t)}_1^2 = 0\,,
\end{equation}
and hence
\begin{equation}\label{e:energyStupid}
  \norm{\theta_s(t)} \leq e^{-\nu \lambda_1 (t - s)} \norm{\theta_{s,0}}\,.
\end{equation}
Our interest, again, is to to investigate how this decay rate can be quantifiably improved when the flow of~$u$ is mixing.
Similar to our treatment of pulsed diffusions, we define the \emph{dissipation time} of~$u$ by
\begin{align*}
  \tau_d &\defeq
    \sup_{s \in \R} \paren[\Big]{
      \inf \set[\Big]{ t - s \st t \geq s, \text{ and }\norm{\theta_s(t)} \leq \frac{\norm{\theta_{s,0}}}{e} \quad\text{for all } \theta_{s,0} \in L^2_0 }
    }
    \\
    &\stackrel{\scriptscriptstyle\phantom{\textup{def}}}{=}
      \sup_{s \in \R} \paren[\Big]{
	\inf \set[\Big]{ t - s \st t \geq s, \text{ and } \norm{\mathcal S_{s, t}}_{L^2_0 \to L^2_0} \leq \frac{1}{e} }
      }\,,
\end{align*}
where $\mathcal S_{s, t}$ is the solution operator to~\eqref{e:ad}.

From~\eqref{e:energyStupid} we immediately see that for any smooth divergence free advecting field~$u$ we again have
\begin{equation*}
  \tau_d \leq \frac{1}{\nu \lambda_1}\,,
\end{equation*}
where $\lambda_1$ is the smallest non-zero eigenvalue of $-\lap$ on $M$.
If the flow of~$u$ is mixing, then we expect that~$\tau_d$ to be much smaller than than $1 / (\lambda_1 \nu)$.
It turns out that all stationary vector fields for which $\nu \tau_d \to 0$ can be elegantly characterized in terms of the spectrum of the operator~$u \cdot \grad$.
Indeed, seminal work of Constantin et.\ al.~\cite{ConstantinKiselevEA08} shows%
\footnote{
  More precisely, in~\cite{ConstantinKiselevEA08} the authors show that an incompressible, time independent, vector field~$u$ is \emph{relaxation enhancing} if and only if $(u \cdot \grad)$ has no eigenfunctions in~$\dot H^1$.
  It is, however, easy to see that a vector field is relaxation enhancing if and only if~$\nu \tau_d \to 0$.
}
that for time independent incompressible vector fields~$u$, $\nu \tau_d \to 0$ if and only if the operator $(u \cdot \grad)$ has no eigenfunctions in~$\dot H^1$.
Consequently, it follows that if the flow generated by~$u$ is weakly mixing, we must have~$\nu \tau_d \to 0$ as $\nu \to 0$.

Our aim is to obtain bounds on the rate at which $\nu \tau_d \to 0$, under an assumption on the rate at which the flow of~$u$ mixes.
The analog of Definition~\ref{d:mixrate} in continuous time is as follows.

\begin{definition}\label{d:mixrateCts}
  Let $h \colon [0, \infty) \to (0, \infty)$ be a continuous, decreasing function that vanishes at $\infty$, and~$\alpha, \beta \geq 0$.
  Let $\varphi_{s, t} \colon M \to M$ be the flow map of~$u$ defined by
  \begin{equation*}
    \partial_t \varphi_{s, t} = u(\varphi_{s, t}, t)
    \qquad\text{and}\qquad
    \varphi_{s, s} = \mathrm{Id}\,.
  \end{equation*}
  \begin{enumerate}\reqnomode
    \item
      We say that the vector field~$u$ is \emph{strongly $\alpha$, $\beta$ mixing with rate function~$h$} if for all $f \in \dot H^\alpha$, $g \in \dot H^\beta$ we have
      \begin{equation}\label{e:hmixrateStrongCts}
	\abs[\big]{\ip{f \circ \varphi_{s,t}, g}}
	  \leq h\paren{t - s} \norm{f}_{\alpha}  \norm{g}_{\beta}\,.
      \end{equation}

    \item
      We say that $\varphi$ is \emph{weakly $\alpha$, $\beta$ mixing with rate function~$h$} if for all $f \in \dot H^\alpha$, $g\in \dot H^\beta$ we have
      \begin{equation}\label{e:hmixrateWeakCts}
	\paren[\Big]{\frac{1}{t - s} \int_s^t
	  \abs[\big]{\ip{f \circ \varphi_{s,r},g}}^2 \, dr}^{1/2}
	  \leq h\paren{t - s} \norm{f}_{\alpha}  \norm{g}_{\beta}\,.
      \end{equation}
  \end{enumerate}
\end{definition}

Our first result bounds the dissipation time of vector fields $u$ that are strongly $\alpha, \beta$ mixing.
\begin{theorem}\label{t:disStrongCts}
  Let $\alpha, \beta > 0$, and $h \colon [0,\infty) \to (0, \infty)$ be a decreasing function that vanishes at infinity.
  If $u$ is strongly $\alpha$, $\beta$ mixing with rate function~$h$, then
  the dissipation time is bounded by 
  \begin{equation}\label{e:taudStrongCts}
    \tau_d \leq \frac{C}{\nu H_3(\nu)}\,.
  \end{equation}
  Here $C$ is a universal constant which can be chosen to be~$18$, and $H_3\colon (0, \infty) \to (0, \infty)$ is defined by
  \begin{align}\label{e:H3}
    H_3(\mu)=\sup \set[\Big]{ \lambda \st \frac{\lambda\exp\Big(4\norm{\nabla u}_{L^\infty} h^{-1}(\frac{1}{2}\lambda^{-\frac{\alpha+\beta}{2}})\Big)}{h^{-1}(\frac{1}{2}\lambda^{-\frac{\alpha+\beta}{2}})}\leq \frac{\norm{\nabla u}_{L^\infty}^2}{2 \mu} }\,,
  \end{align}
  where $h^{-1}$ is the inverse function of $h$.
  \end{theorem}

  As before, we now compute $H_3$ explicitly for polynomial, and exponential rate functions.
  
  \begin{corollary}\label{c:SMixExp}
    Let~$\alpha, \beta, u, h$ be as in Theorem~\ref{t:disStrongCts}.
    \begin{enumerate}\reqnomode
      \item If the mixing rate function~$h$ is the power law~\eqref{e:powerh}, then
	\begin{equation}\label{e:SMixExpPower}
	  \tau_d \leq \frac{C}{\nu \abs{\ln \nu}^\delta}\,,
	  \qquad\text{where}\quad
	  \delta \defeq \frac{2p}{\alpha+\beta}\,,
	\end{equation}
	and $C = C(\alpha, \beta, c, \norm{\grad u}_{L^\infty})$ is a finite constant.
      \item If the mixing rate function~$h$ is the exponential~\eqref{e:exp}, then
	\begin{equation}\label{e:SMixExpExp}
	  \tau_d \leq \frac{C}{\nu^\delta}\,,
	  \qquad\text{where}\qquad
	  \delta \defeq \frac{2(\alpha+\beta)\norm{\nabla u}_{L^\infty}}{c_2+2(\alpha+\beta)\norm{\nabla u}_{L^\infty}}\,,
	\end{equation}
	and $C = C(\alpha, \beta, c_1,c_2, \norm{\grad u}_{L^\infty})$ is a finite constant.
    \end{enumerate}
  \end{corollary}
  \begin{remark}
    The cases considered in Corollary~\ref{c:SMixExp} were also recently studied by Coti Zelati, Delgadino and Elgindi~\cites{CotiZelatiDelgadinoEA18}.
    Here the authors show that if the mixing rate is given by the power law~\eqref{e:powerh}, then the dissipation time is bounded by
    \begin{equation*}
      \tau_d \leq \frac{C}{\nu^\delta}\,,
      \qquad\text{where}\quad
      \delta = \frac{\alpha+\beta}{\alpha+\beta + p}\,.
    \end{equation*}
    Alternately, if the mixing rate is the exponential~\eqref{e:exp}, then~\cite{CotiZelatiDelgadinoEA18} show that the dissipation time is bounded by
    \begin{equation*}
      \tau_d \leq C\abs{\ln \nu}^2\,.
    \end{equation*}
    In both these cases, the bounds provided by~\cite{CotiZelatiDelgadinoEA18} are stronger than those provided by Corollary~\ref{c:SMixExp}.
  \end{remark}

Next we bound the dissipation time for weakly mixing flows.

\begin{theorem}\label{t:disWeakCts}
  Let $\alpha, \beta > 0$, and $h \colon [0,\infty) \to (0, \infty)$ be a decreasing function that vanishes at infinity.
  If $u$ is strongly $\alpha$, $\beta$ mixing with rate function~$h$, then
  the dissipation time is bounded by 
  \begin{equation}\label{e:taudStrongCtsWeak}
    \tau_d \leq  \frac{C}{\nu H_4(\nu)} \,.
  \end{equation}
  Here $C$ is a universal constant which can be chosen to be~$18$, and $H_4\colon (0, \infty) \to (0, \infty)$ is defined by
  \begin{align}\label{e:H4}
    H_4(\mu)=\sup \set[\Big]{ \lambda \st \frac{\lambda\exp\Big(4\norm{\nabla u}_{L^\infty} h^{-1}(\frac{1}{2\sqrt{\tilde c}}\lambda^{-\sfrac{(d+2\alpha+2\beta)}{4}})\Big)}{h^{-1}(\frac{1}{2\sqrt{\tilde c}}\lambda^{-\sfrac{(d+2\alpha+2\beta)}{4}})}\leq \frac{\norm{\nabla u}_{L^\infty}^2}{2\mu} }\,,
  \end{align}
  where $h^{-1}$ is the inverse function of $h$ and $\tilde c = \tilde c(M) > 0$ is the same constant as in Theorem~\ref{t:disWeak} and Remark~\ref{r:ctilde}.
  \end{theorem}

As before, we compute the above dissipation time bound explicitly when the mixing rate function decays polynomially.

\begin{corollary}\label{c:ctsweak}
  Suppose~$u$ is weakly $\alpha$, $\beta$ mixing with rate function~$h$, where $\alpha, \beta>0$, and h is power law~\eqref{e:powerh}.
  Then the dissipation time is bounded by
  \begin{equation}\label{e:ctsweak}
    \tau_d \leq \frac{C}{\nu \abs{\ln\nu}^{\delta}}\,,
    \qquad\text{where}\quad
    \delta = \frac{4p}{d+2\alpha+2\beta}\,,
  \end{equation}
  and $C = C(c,\tilde c,\alpha, \beta, \norm{\grad u}_{L^\infty})$ is some finite constant.
\end{corollary}

\begin{remark}[Comparison with pulsed diffusions]
In continuous time, the estimate on the dissipation time~\eqref{e:taudStrongCts} is \emph{weaker} than that of a pulsed diffusion, with the same mixing rate function.
In particular, if $h$ decays algebraically, then $\nu \tau_d$ decays algebraically for pulsed diffusions (as in Corollary~\ref{c:strong}) but only logarithmically (as in Corollary~\ref{c:SMixExp}) for the advection diffusion equation.
The reason our method yields a stronger results for pulsed diffusions is because because pulsed diffusions are better approximated by the underlying dynamical system than solutions to~\eqref{e:ad} are.
Thus when studying pulsed diffusions one is able to better use the mixing properties of the underlying dynamical system.
\end{remark}

\begin{remark}[Shear Flows]
  In the particular case of shear flows a stronger estimate on the dissipation time can be obtained using Theorem~1.1 in~\cite{BedrossianCotiZelati17}.
  Namely let $u = u(y)$ be a smooth shear flow on the $2$-dimensional torus with non-degenerate critical points, and let $L^2_0$ denote the space of all functions whose horizontal average is~$0$.
  Now Theorem~1.1 in~\cite{BedrossianCotiZelati17} guarantees that the dissipation time is bounded by
  \begin{equation}\label{e:shear}
    \tau_d \leq C \frac{\abs{\ln \nu}^2}{\nu^{1/2}}\,,
  \end{equation}
  for some constant $C > 0$.

  To place this in the context of our results, we restrict our attention to $L^2_0$ functions on~$\T^2$ whose horizontal averages are all~$0$.
  On this space, the method of stationary phase can be used to show that the flow generated by~$u$ is strongly $1$, $1$ mixing with rate function $h(t) = C t^{-1 / 2}$ (see equation~(1.8) in~\cite{BedrossianCotiZelati17}).
  Consequently, by Corollary~\ref{c:SMixExp} guarantees that the dissipation time is bounded by 
  \begin{equation*}
    \tau_d \leq \frac{C}{\nu \abs{\ln \nu}^\delta}\,,
    \qquad\text{where}\quad
    \delta =\frac{2p}{\alpha+\beta}.
  \end{equation*}
  This, however, is weaker than~\eqref{e:shear}.
\end{remark}

\begin{remark}[Optimality]\label{r:optimality}
We recall that Poon~\cite{Poon96} (see also~\cite[eq. 9]{MilesDoering18}) showed the double exponential \emph{lower bound}
\begin{equation}\label{e:dexpLowerCts}
  \norm{\theta_s(t)} \geq \exp\paren[\Big]{
    -\frac{\nu C \norm{u}_{C^1} \norm{\theta_0}_1^2}{\norm{\theta_0}^2}
    \gamma^{t - s}
  } \norm{\theta_{s,0}} \,,
\end{equation}
for some constants~$C > 0$ and $\gamma > 1$.
To the best of our knowledge, there are no incompressible smooth divergence free vector fields for which the lower bound~\eqref{e:dexpLowerCts} is attained.
Moreover, on the torus, recent work of Miles and Doering~\cite{MilesDoering18} suggests that the Batchelor length scale may limit the long term effectiveness of mixing forcing only a single-exponential energy decay.

As with the case of pulsed diffusions Remark~\ref{r:taudLog}, the lower bound~\eqref{e:dexpLowerCts} implies that the dissipation time is again bounded below by $O(\abs{\ln \nu})$ as in~\eqref{e:taudLogLower}.
The upper bounds currently available are either algebraic (Corollary~\ref{c:SMixExp}), or $O(\abs{\ln \nu}^2)$ (as in~\cite{CotiZelatiDelgadinoEA18}).
Thus there is a gap between the currently available upper and lower bounds on the dissipation time.
Moreover, while we are able to exhibit pulsed diffusions that have a logarithmic dissipation time (Theorem~\ref{t:energydecay} and Remark~\ref{r:taudLog}), we do not know examples of smooth flows whose dissipation time is $O(\abs{\ln \nu})$.
\end{remark}
\medskip

Finally, we turn our attention to studying the principal eigenvalue of the operator~$-\nu \lap + u \cdot \grad$ in a bounded domain~$\Omega$ with Dirichlet boundary conditions.
In this case, in addition to~$u$ being smooth and divergence free, we also assume~$u$ is time independent and tangential on the boundary (i.e.~$u \cdot \hat n = 0$ on $\partial \Omega$, where $\hat n$ denotes the outward pointing unit normal).
Let $\mu_0(\nu, u)$ denote the principal eigenvalue of $-\nu \lap + u \cdot \grad$ with homogeneous Dirichlet boundary conditions on~$\partial \Omega$.

By Rayleigh's principle we note
\begin{equation*}
  \mu_0(\nu, u) \geq \mu_0( \nu, 0 ) = \nu \mu_0(1, 0)
\end{equation*}
where $\mu_0(1, 0)$ is the principal eigenvalue of the Laplacian.
Our interest is in understanding the behaviour of $\mu_0(\nu, u) / \nu$ as $\nu \to 0$.
Berestycki et.\ al.~\cite{BerestyckiHamelEA05} showed that $\mu_0(\nu, u) / \nu \to \infty$ if and only if $u \cdot \grad$ has no first integrals in $H^1_0$.
That is, $\mu_0(\nu, u) / \nu \to \infty$ if and only if there does not exist $w \in H^1_0(\Omega)$ such that $u \cdot \grad w = 0$.

In general it does not appear to be possible to obtain a rate at which $\mu_0(\nu, u) / \nu \to \infty$.
If, however, the flow generated by~$u$ is sufficiently mixing then we obtain a rate at which $\mu_0(\nu, u) / \nu \to \infty$ in terms of the mixing rate of~$u$.
This is our next result.
\begin{proposition}\label{p:eval}
  If~$u$ is a smooth, time independent, incompressible vector field which is tangential on~$\partial \Omega$, then
  \begin{equation}\label{e:lamda0taud}
    \frac{\mu_0(\nu, u)}{\nu} \geq \frac{1}{\nu \tau_d}\,.
  \end{equation}
\end{proposition}

Proposition~\ref{p:eval} follows immediately by solving the advection diffusion equation with the principal eigenfunction as the initial data.
For completeness we present the proof in Section~\ref{s:peval}.

  Now we note the proof of Theorems~\ref{t:disStrongCts}, \ref{t:disWeakCts} only use the spectral decomposition of the Laplacian, and are unaffected by the presence of spatial boundaries.
  Thus Theorems~\ref{t:disStrongCts} and~\ref{t:disWeakCts} still apply in this context.
  Consequently, if~$u$ is known to be (strongly, or weakly) mixing at a particular rate, then $\mu_0(\nu, u) / \nu$ must diverge to infinity, and the growth rate can be obtained by using~\eqref{e:lamda0taud} and Theorems~\ref{t:disStrongCts},~\ref{t:disWeakCts}, or Corollaries~\ref{c:SMixExp}, \ref{c:ctsweak} as appropriate.

For example, if~$\alpha, \beta > 0$ and~$u$ is strongly $\alpha$, $\beta$ mixing with the exponentially decaying rate function~\eqref{e:exp}, then
\begin{equation}\label{e:lambdaLowerBound}
  \frac{\mu_0(\nu, u)}{\nu} \geq \frac{1}{C\nu^\gamma}\,,
  \qquad\text{where}\qquad
  \gamma = \frac{c_2}{c_2+2(\alpha+\beta)\norm{\nabla u}_{L^\infty}} \,,
\end{equation}
and $C = C(\alpha, \beta, h)$ is a finite constant.
Using~\cite{CotiZelatiDelgadinoEA18}, this can be improved to the bound
\begin{equation*}
  \frac{\mu_0( \nu, u)}{\nu} \geq \frac{1}{C \nu \abs{\ln \nu}^2 }\,.
\end{equation*}
We remark, however, that in view of Remark~\ref{r:optimality} and~\eqref{e:lamda0taud}, we expect that if~$u$ that generates an exponentially mixing flow, then one should have
\begin{equation*}
  \frac{\mu_0(\nu, u)}{\nu} \geq \frac{1}{C \nu \abs{\ln \nu}}\,.
\end{equation*}
We are, however, presently unable to prove this stronger bound.

The rest of this paper is devoted to the proofs of the main results.
A brief plan can be found at the end of Section~\ref{s:intro}.

\section{Dissipation Enhancement for Pulsed Diffusions.}\label{s:discrete}

In this section we prove Theorems~\ref{t:disStrong} and~\ref{t:disWeak}.
The main idea behind the proof is to split the analysis into two cases.
In the first case, we assume $\norm{\theta_n}_1 / \norm{\theta_n}$ is large, and obtain decay of~$\norm{\theta_n}$ using the energy inequality.
In the second case, $\norm{\theta_n}_1 / \norm{\theta_n}$ is small, and hence the dynamics are well approximated by that of the underlying dynamical system.
The mixing assumption now forces the generation of high frequencies, and the rapid dissipation of these gives an enhanced decay of $\norm{\theta_n}$.

\subsection{The Strongly Mixing Case.}
We begin by stating two lemmas handling each of the cases stated above.

\begin{lemma}\label{l:H1ByL2Large}
  Given $\theta \in L^2_0$, define $\Enu \theta$ by
  \begin{equation}\label{e:EnuDef}
    \Enu \theta \defeq \frac{1}{\nu} \norm[\big]{ \paren{1 - e^{2 \nu \lap}}^{1/2} U \theta}^2\,.
  \end{equation}
  If for $\theta_0 \in L^2_0$ and $c_0 > 0$ we have
  \begin{equation}\label{e:DnuLower}
    \Enu \theta_0
      \geq c_0 \norm{\theta_0}^2\,,
  \end{equation}
  then
  \begin{equation*}
    \norm{\theta_{1}}^2 \leq e^{-\nu c_0} \norm{\theta_0}^2\,.
  \end{equation*}
\end{lemma}

\begin{lemma}\label{l:H1ByL2Small}
  Let $0 < \lambda_1 < \lambda_2 \leq\cdots$ be the eigenvalues of the Laplacian, where each eigenvalue is repeated according to its multiplicity.
 Let $\lambda_N$ be the largest eigenvalue satisfying $\lambda_N\leq H_1(\nu)$, where we recall that $H_1$ is defined in~\eqref{e:H1}.
  If
  \begin{equation}\label{e:EnuUpper}
    \Enu \theta_0
      < \lambda_N \norm{\theta_0}^2\,,
  \end{equation}
  then for
  \begin{equation}\label{e:m0choice}
    m_0 = 2 \ceil[\Big]{
      h^{-1} \paren[\Big]{ \frac{1}{2}\lambda_N^{-\sfrac{(\alpha+\beta)}{2}}}
    }
  \end{equation}
  and all sufficiently small~$\nu > 0$, we have
  \begin{equation}\label{e:thetaM0Upper}
    \norm{ \theta_{m_0} }^2 \leq \exp\paren[\Big]{- \frac{\nu H_1(\nu) m_0}{16} } \norm{\theta_0}^2\,.
  \end{equation}
  Here $h^{-1}$ is the inverse function of $h$.
\end{lemma}

Momentarily postponing the proofs of Lemmas~\ref{l:H1ByL2Large} and~\ref{l:H1ByL2Small} we prove Theorem~\ref{t:disStrong}.
\begin{proof}[Proof of Theorem~\ref{t:disStrong}]
  Choosing $c_0 = \lambda_N$ and repeatedly applying Lemmas~\ref{l:H1ByL2Large} and~\ref{l:H1ByL2Small} we obtain an increasing sequence of times $n_k$  such that
  \begin{equation*}
    \norm{\theta_{n_k}} ^2\leq 
    \exp\paren[\Big]{- \frac{\nu H_1(\nu) n_k}{16} } \norm{\theta_0}^2\,,
      \qquad
      \text{and}\qquad
      n_{k+1} - n_k \leq m_0\,.
  \end{equation*}
  This immediately implies
  \begin{equation}\label{e:taud1}
    \tau_d \leq \frac{32}{\nu H_1(\nu)} + m_0\,.
  \end{equation}
  Note by choice of $\lambda_N$ we have
  \begin{equation*}
    h\paren[\Big]{\frac{1}{2\sqrt{\nu\lambda_N }}}\leq \frac{\lambda_N^{-(\alpha+\beta)/2}}{2}\,.
  \end{equation*}
  And since $h$ is decreasing, it further implies
  \begin{align*}
  h^{-1}\paren[\Big]{\frac{\lambda_N^{-(\alpha+\beta)/2}}{2}}\leq \frac{1}{2\sqrt{\nu\lambda_N}}\,.
  \end{align*}
 By the choice of $m_0$, we then have
  \begin{equation}\label{e:m0tmp}
    m_0 \leq \frac{1}{\sqrt {\nu\lambda_N}}\leq \frac{1}{\nu\lambda_N} \,.
  \end{equation}
 Recall by Weyl's lemma (see for instance~\cite{MinakshisundaramPleijel49}) we know
  \begin{equation}\label{e:weyl}
    \lambda_j \approx \frac{4\pi\, \Gamma(\frac{d}{2}+1)^{\sfrac{2}{d}}}{\vol(M)^{\sfrac{2}{d}}}j^{\sfrac{2}{d}}\,,
  \end{equation}
  asymptotically as $j \to \infty$.
  This implies $\lambda_{j+1}-\lambda_{j}=o(\lambda_j)$.
  Using this, and the fact that $H_1(\nu)\to \infty$ as $\nu\to 0$, we must have
  \begin{equation}\label{e:lambdahnu}
   \frac{1}{2}H_1(\nu)\leq  \lambda_N\leq H_1(\nu)\,,
  \end{equation}
  when $\nu$ is sufficiently small.
  Substituting this in~\eqref{e:m0tmp} gives
  \begin{align*}
  m_0\leq \frac{2}{\nu H(\nu)}\,,
  \end{align*} 
  and using this in~\eqref{e:taud1} yields the desired result.
\end{proof}

To prove Corollary~\ref{c:strong}, we only need to compute the function $H_1$ explicitly for the specific rate functions of interest.

\begin{proof}[Proof of Corollary~\ref{c:strong}]\label{pg:c:strong}
When the mixing rate function~$h$ is the power law as defined in~\eqref{e:powerh}, we compute
\begin{align*}
  H_1(\nu)
    =\paren[\Big]{\frac{4^{p-1}}{c^2 \nu^p}}^{\frac{1}{\alpha+\beta+p}}\,.
\end{align*}
Substituting this into~\eqref{e:taud} yields~\eqref{e:taudstrongPower} as desired.

When the mixing rate function~$h$ is the exponential function as defined in~\eqref{e:exp}, we can not compute~$H_1$ exactly, as~\eqref{e:H1} only yields
\begin{equation}\label{e:H1bound1}
H_1(\nu) = \frac{c_2^2}{4\nu}\paren[\Big]{\ln 2+\ln{ c_1}+\frac{\alpha+\beta}{2}\ln{H_1(\nu)} }^{-2}\,.
\end{equation}
Since $H_1(\nu) \to \infty$ as $\nu \to 0$, we know $H_1(\nu) \geq 1$ for sufficiently small~$\nu$.
\begin{equation*}
  H_1(\nu) \leq \frac{C}{\nu}\,,
\end{equation*}
for some constant $C = C(c_1,c_2,\alpha,\beta)$.
Using this in~\eqref{e:H1bound1} yields
\begin{align*}
H_1(\nu)\geq \frac{C}{\nu\abs{\ln \nu}^2}\,.
\end{align*}
Substituting this in~\eqref{e:taud} yields~\eqref{e:taudstrongExp} as desired.
This argument can also be iterated to obtain improved bounds as stated in Remark~\ref{r:taudExpImproved}.
\end{proof}

It remains to prove Lemmas~\ref{l:H1ByL2Large} and~\ref{l:H1ByL2Small}.

\begin{proof}[Proof of Lemma~\ref{l:H1ByL2Large}]
  Let $\set{e_i}$ be a Hilbert basis of $L^2_0$ with $-\lap e_i = \lambda_i e_i$.
  Note that~\eqref{e:pulseddiff2} and~\eqref{e:EnuDef} imply the energy equality
  \begin{align}
    \nonumber
      \norm{\theta_{1}}^2 &= \sum_{i=1}^\infty e^{-2\nu \lambda_i}\abs{\ip{U\theta_0,e_i}}^2 = \sum_{i=1}^\infty \abs{\ip{U\theta_0,e_i}}^2 - \nu \Enu\theta_0
    \\
    \label{e:l2decay1}
      &= \norm{\theta_0}^2 - \nu \Enu\theta_0\,.
  \end{align}
  Now using~\eqref{e:DnuLower} immediately implies
  \begin{equation}
    \norm{\theta_{1}}^2 \leq (1-c_0\nu)\norm{\theta_0}^2\leq e^{-c_0\nu}\norm{\theta_0}^2\,.\qedhere
  \end{equation}
\end{proof}

In order to prove Lemma~\ref{l:H1ByL2Small}, we first need to estimate the difference between the pulsed diffusion and the underlying dynamical system.
We do this as follows.
\begin{lemma}\label{invisdiff}
  Let $\phi_n$, defined by
  \begin{equation*}
    \phi_n = U^n \theta_0\,,
  \end{equation*}
  be the evolution of $\theta_0$ under the dynamical system generated by~$\varphi$.
  Then for all $n \geq 0$ we have
  \begin{equation}\label{e:ThetaNMinusPhiN}
    \norm{\theta_n-\phi_n} \leq \sum_{k=0}^{n-1}\sqrt{\nu \Enu\theta_k}\,.
  \end{equation}
\end{lemma}
\begin{proof}
  Since $\phi_n = U \phi_{n-1}$, we have
  \begin{align*}
    \norm{\theta_n-\phi_n}
      &\leq \norm{(e^{\nu\lap}-1)U\theta_{n-1}}
	+\norm{U \paren{{\theta_{n-1}-\phi_{n-1}}}}
    \\
      &= \paren[\Big]{\sum_{i=1}^\infty (e^{-\nu\lambda_i}-1)^2\abs{\ip{U\theta_{n-1},e_i}}^2}^{1/2}+\norm{\theta_{n-1}-\phi_{n-1}} \\
    &\leq \paren[\Big]{\sum_{i=1}^\infty (1-e^{-2\nu\lambda_i})\abs{\ip{U\theta_{n-1},e_i}}^2}^{1/2}+\norm{\theta_{n-1}-\phi_{n-1}} \\
    &\leq \sqrt{\nu \Enu \theta_{n-1}}+\norm{\theta_{n-1}-\phi_{n-1}}\,,
  \end{align*}
  and hence~\eqref{e:ThetaNMinusPhiN} follows by induction.
\end{proof}

We now prove Lemma~\ref{l:H1ByL2Small}.
\begin{proof}[Proof of Lemma~\ref{l:H1ByL2Small}]
  By~\eqref{e:l2decay1}, we have
  \begin{align}\label{e:energybasic}
    \norm{\theta_{m_0}}^2=\norm{\theta_1}^2-\nu\sum_{m=1}^{m_0-1}\Enu\theta_{m}\,.
  \end{align}
  Thus the decay of $\norm{\theta_{m_0}}$ is governed by the growth of $\sum_{m=1}^{m_0-1}\Enu\theta_{m}$.
  In order to estimate $\Enu \theta_m$ we claim
  \begin{align}\label{e:epsilonH1Relation}
    2\norm{\theta_{m+1}}_1^2 \leq \Enu\theta_m  \leq 2\norm{U\theta_m}_1^2\,, \quad \text{for all } m \in \N\,.
  \end{align}
  Indeed, by definition of~$\Enu$ (equation~\eqref{e:EnuDef}) we have
  \begin{equation*}
    \nu \Enu \theta_m
      = \sum_{k=1}^\infty \paren[\big]{ 1 - e^{-2 \nu \lambda_k} } \abs{(U \theta_m)^\wedge(k)}^2\,,
  \end{equation*}
  where $(U\theta_m)^\wedge(k) \defeq \ip{U\theta_m, e_k}$ is the $k$-th Fourier coefficient of $U\theta_m$, and $\set{e_k}$ is a Hilbert basis of $L^2_0$ with $-\lap e_k = \lambda_k e_k$.
  Now~\eqref{e:epsilonH1Relation} follows from the inequalities
  \begin{equation*}
    2\nu \lambda_k e^{-2\nu \lambda_k} 
      \leq 1 - e^{-2 \nu \lambda_k}
      \leq 2 \nu \lambda_k \,.
  \end{equation*}

  We next claim that for all sufficiently small~$\nu$ we have
  \begin{align}\label{H1small}
    \norm{\theta_1}_1^2< \lambda_N\norm{\theta_1}^2\,.
  \end{align}
  To see this, note that~\eqref{e:EnuUpper} and~\eqref{e:epsilonH1Relation} imply
  \begin{equation}\label{e:tmp1}
    \norm{\theta_1}_1^2\leq \frac{1}{2}\Enu \theta_0< \frac{\lambda_N}{2}\norm{\theta_0}^2\,.
  \end{equation}
  Moreover, our choice of~$\lambda_N$ (in equation~\eqref{e:H1}) guarantees~$\lambda_N \leq 1 / (2\nu)$ for all~$\nu$ sufficiently small.
  Thus
  \begin{align*}
    \norm{\theta_1}^2=\norm{\theta_0}^2-\nu\Enu\theta_0\geq (1-\nu\lambda_N)\norm{\theta_0}^2\geq \frac{1}{2}\norm{\theta_0}^2\,,
  \end{align*}
  and substituting this in equation~\eqref{e:tmp1} gives~\eqref{H1small} as claimed.

  We now claim that for~$N$ and $m_0$ as in the statement of Lemma~\ref{l:H1ByL2Small} we have
  \begin{equation}\label{e:tmp2}
  \sum_{m=1}^{m_0-1}\Enu\theta_{m}\geq \frac{\lambda_Nm_0}{8}\norm{\theta_1}^2\,.
  \end{equation}
  Note equation~\eqref{e:tmp2} immediately implies~\eqref{e:thetaM0Upper}.
  Indeed, by~\eqref{e:energybasic}, we have
  \begin{align*}
    \norm{\theta_{m_0}}^2
      &\leq \paren[\Big]{1-\frac{\nu\lambda_Nm_0}{8}}\norm{\theta_1}^2
      \leq \exp\paren[\Big]{-\frac{\nu\lambda_Nm_0}{8}}\norm{\theta_0}^2
    \\
      &\leq \exp\paren[\Big]{-\frac{\nu H_1(\nu)m_0}{16}}\norm{\theta_0}^2\,,
  \end{align*}
  where last inequality followed from~\eqref{e:lambdahnu}.

  Thus it only remains to prove equation~\eqref{e:tmp2}.
  For this we let $\phi_m$, defined by
  \begin{equation*}
    \phi_m = U^{m-1} \theta_1\,,
  \end{equation*}
  be the evolution of~$\theta_1$ under the dynamical system generated by~$\varphi$.
  Let $P_N \colon L^2_0 \to L^2_0$
  be the orthogonal projection onto $\operatorname{span}\set{e_1, \dots, e_N}$.
  Using~\eqref{e:epsilonH1Relation} we have
  \begin{align}
    \nonumber
    \MoveEqLeft
    \sum_{m=1}^{m_0-1}\Enu\theta_{m}
    \geq \sum_{m=m_0/2}^{m_0-1}\Enu\theta_{m}
    \geq 2\sum_{m=m_0/2}^{m_0-1}\norm{\theta_{m+1}}_1^2
    \\
    \nonumber
    &\geq 2\lambda_N\sum_{m=m_0/2}^{m_0-1}\norm{(I-P_N)\theta_{m+1}}^2
    \\
    \nonumber
    &\geq \lambda_N\Bigl( \sum_{m=m_0/2}^{m_0-1}\norm{(I-P_N)\phi_{m+1}}^2\\
    \nonumber
    &\qquad\qquad
      -2 \sum_{m=m_0/2}^{m_0-1}\norm{(I-P_N)(\theta_{m+1}-\phi_{m+1})}^2 \Bigr)
    \\
    \label{e:energy1}
    &\geq \lambda_N \Big(\frac{m_0}{2}\norm{\phi_1}^2-\sum_{m=m_0/2}^{m_0-1}\norm{P_N\phi_{m+1}}^2-2\sum_{m=m_0/2}^{m_0-1}\norm{\theta_{m+1}-\phi_{m+1}}^2\Big)\,.
  \end{align}

  Now using Lemma~\ref{invisdiff} we estimate the last term on the right of~\eqref{e:energy1} by
  \begin{align}\label{energysub1}
    \nonumber
    \sum_{m=m_0/2}^{m_0-1}\norm{\theta_{m+1}-\phi_{m+1}}^2 
    &\leq \sum_{m=m_0/2}^{m_0-1}\paren[\Big]{\sum_{l=1}^{m}\sqrt{\nu \Enu\theta_{l}}}^2 
    \leq  \sum_{m=m_0/2}^{m_0-1}m\nu \sum_{l=1}^{m}\Enu\theta_{l}\\
    &\leq \frac{ m_0^2\nu}{2}\sum_{l=1}^{m_0-1}\Enu\theta_{l}\,.
  \end{align}

  For the second term on the right of~\eqref{e:energy1} we note that since $U$ is strongly $\alpha,\beta$ mixing with rate function $h$, we have
  \begin{equation*}
    \norm{U^m f}_{-\beta} \leq h(m) \norm{f}_\alpha\,,
  \end{equation*}
  for every $f \in \dot H^\alpha$ (see also~\eqref{e:HminusBeta} in Appendix~\ref{s:mixrates}).
  This implies
  \begin{align}\label{energysub2}
    \nonumber
    \MoveEqLeft[5]
    \sum_{m=m_0/2}^{m_0-1}\norm{P_N\phi_{m+1}}^2\leq \sum_{m=m_0/2}^{m_0-1}\lambda_N^\beta\norm{\phi_{m+1}}_{-\beta}^2
    \leq \sum_{m=m_0/2}^{m_0-1}\lambda_N^\beta h(m)^2 \norm{\phi_1}_\alpha^2\\
    \nonumber
   & \leq m_0 h\paren[\Big]{ \frac{m_0}{2}}^2\lambda_N^\beta\norm{\phi_1}_\alpha^2
    \leq m_0 h\paren[\Big]{ \frac{m_0}{2}}^2\lambda_N^\beta \norm{\theta_1}^{2-2\alpha}\norm{\theta_1}_1^{2\alpha}
    \\
   & \leq  m_0 h\paren[\Big]{ \frac{m_0}{2}}^2\lambda_N^{\alpha+\beta}\norm{\theta_1}^2 \,,
  \end{align}
  where the last inequality followed from \eqref{H1small}.
  
  Substituting \eqref{energysub1} and \eqref{energysub2} in \eqref{e:energy1} we obtain
  \begin{equation}\label{e:enutmp1}
    \sum_{m=1}^{m_0-1}\Enu\theta_{m}\geq \frac{m_0\lambda_N}{1+\lambda_N\nu m_0^2}\paren[\Big]{\frac{1}{2}- h\paren[\Big]{ \frac{m_0}{2}}^2\lambda_N^{\alpha+\beta}}\norm{\theta_1}^2\,.
  \end{equation}
  Clearly, by choice of~$m_0$ in~\eqref{e:m0choice}, we know
  \begin{equation}\label{e:hm0tmp1}
    h\paren[\Big]{ \frac{m_0}{2}}^2 \lambda_N^{\alpha + \beta} \leq \frac{1}{4}\,.
  \end{equation}
  Moreover, using the definition of~$H_1$~\eqref{e:H1} and the fact that $\lambda_N \leq H_1(\nu)$, we see
  \begin{equation}\label{e:lanmo2tmp1}
    \lambda_N \nu m_0^2 \leq 1\,.
  \end{equation}
  Now using~\eqref{e:hm0tmp1} and~\eqref{e:lanmo2tmp1} in~\eqref{e:enutmp1} implies~\eqref{e:tmp2}.
  This finishes the proof of Lemma~\ref{l:H1ByL2Small}.
\end{proof}

\subsection{The Weakly Mixing Case.}

We now turn our attention to Theorem~\ref{t:disWeak}.
The proof is very similar to the proof of Theorem~\ref{t:disStrong}, the only difference is that the analog of Lemma~\ref{l:H1ByL2SmallWeak} is not as explicit.

\begin{lemma}\label{l:H1ByL2SmallWeak}
  Let $\lambda_N$ be the largest eigenvalue of $-\lap$ such that  $\lambda_N\leq H_2(\nu)$,  and suppose
  \begin{equation*}
    \Enu \theta_0
      < \lambda_N \norm{\theta_0}^2\,.
  \end{equation*}
  Then, for all sufficiently small~$\nu > 0$, we have
  \begin{equation*}
    \norm{ \theta_{m_0} }^2 \leq \exp\paren[\Big]{- \frac{\nu H_2(\nu) m_0}{16} } \norm{\theta_0}^2\,.
  \end{equation*}
  where
  \begin{align}\label{e:m0disweak}
m_0=2\floor[\Big]{h^{-1}\Big(\frac{1}{2\sqrt{\tilde c}}\lambda_N^{-\sfrac{(d+2\alpha+2\beta)}{4}}\Big)}\,, 
\end{align}
  and~$\tilde c$ is the constant in Theorem~\ref{t:disWeak} and Remark~\ref{r:ctilde}.
\end{lemma}

Given Lemma~\ref{l:H1ByL2SmallWeak}, the proof of Theorem~\ref{t:disWeak} is essentially the same as the proof of Theorem~\ref{t:disStrong}.
 \begin{proof}[Proof of Theorem \ref{t:disWeak}]
  Choosing $c_0 = \lambda_N$ and repeatedly applying Lemmas~\ref{l:H1ByL2Large} and~\ref{l:H1ByL2SmallWeak} we obtain an increasing sequence of times $n_k$  such that
  \begin{equation*}
    \norm{\theta_{n_k}} ^2\leq 
    \exp\paren[\Big]{- \frac{\nu H_2(\nu) n_k}{16} } \norm{\theta_0}^2\,,
      \qquad
      \text{and}\qquad
      n_{k+1} - n_k \leq m_0\,.
  \end{equation*}
  This immediately implies
  \begin{equation}\label{e:taud2}
    \tau_d \leq \frac{32}{\nu H_2(\nu)} + m_0\,.
  \end{equation}
  By the choice of $m_0$ and $\lambda_N$, we notice that 
  \begin{equation*}
  m_0
  \leq \frac{1}{\sqrt{\nu\lambda_N}} \leq  \frac{1}{\nu\lambda_N}\leq \frac{2}{\nu H_2(\nu)}\,.
  \end{equation*}
  This proves~\eqref{e:taudWeak}.
\end{proof}
 
Before proving Lemma~\ref{l:H1ByL2SmallWeak}, we prove Corollary~\ref{c:taudWeakPower}.
 
\begin{proof}[Proof of Corollary~\ref{c:taudWeakPower}]
  The proof only involves computing~$H_2$ explicitly when~$h$ is given by the power law~\eqref{e:powerh}.
  Using~\eqref{e:H2} we see
 \begin{align*}
   H_2(\nu)=\paren[\Big]{
     2^{(p+2)/2}c\sqrt{\tilde c}}^{-4\delta'}
    \nu^{-2p\delta'} 
    \,,
    \qquad\text{where}\quad
    \delta' \defeq \frac{1}{2\alpha + 2\beta + 2p + d} \,.
 \end{align*}
 Substituting this into~\eqref{e:taudWeak} yields~\eqref{e:taudWeakPower} as desired.
 \end{proof}
It remains to prove Lemma~\ref{l:H1ByL2SmallWeak}.

\begin{proof}[Proof of Lemma \ref{l:H1ByL2SmallWeak}]
  We first claim that~\eqref{e:tmp2} still holds if $\lambda_N$, $m_0$ chosen as in the statement of Lemma~\ref{l:H1ByL2SmallWeak}.
  Once~\eqref{e:tmp2} is established, then the remainder of the proof is identical to that of  Lemma~\ref{l:H1ByL2Small}.

  To prove~\eqref{e:tmp2}, we observe that the lower bound~\eqref{e:energy1} (from the proof of Lemma~\ref{l:H1ByL2Small}) still holds in this case.
  For last term on the right of~\eqref{e:energy1}, we use the bound~\eqref{energysub1}.
  The only difference here is to estimate the second term using the weak mixing assumption~\eqref{e:hmixrateWeak} instead.
  Observe
  \begin{align*}
    \frac{1}{m_0}\sum_{m=0}^{m_0-1}\norm{P_N\phi_{m+1}}^2
    &=\sum_{l=1}^N \frac{1}{m_0} \sum_{m=0}^{m_0-1} \abs{ \ip{ e_l, U^m \theta_1} }^2\,.
  \end{align*} 
  Since $\varphi$ is  weak $\alpha, \beta$-mixing with rate function $h$,~\eqref{e:hmixrateWeak} yields
  \begin{align*}
    \frac{1}{m_0} \sum_{m=0}^{m_0-1} \abs{ \ip{e_l, U^m \theta_1} }^2
      \leq h(m_0-1)^2\norm{\theta_1}_{\alpha}^2\lambda_l^\beta
      \leq  h(m_0-1)^2\lambda_N^\beta\norm{\theta_1}_{\alpha}^2
    \\
      \leq h(m_0-1)^2\lambda_N^\beta\norm{\theta_1}^{2-2\alpha}\norm{\theta_1}_1^{2\alpha} 
      \leq  h(m_0-1)^2\lambda_N^{\beta+\alpha}\norm{\theta_1}^{2}\,,
  \end{align*} 
  Note that the last inequality above comes from~\eqref{H1small}.  This gives
  \begin{align*}
    \frac{1}{m_0}\sum_{m=0}^{m_0-1}\norm{P_N\phi_{m+1}}^2
    &
     \leq 
     h(m_0-1)^2N\lambda_N^{\beta+\alpha}\norm{\theta_1}^{2}
     \\
     &
    \leq \tilde ch(m_0-1)^2\lambda_N^{(d+2\alpha+2\beta)/2}\norm{\theta_1}^2\,.
  \end{align*}
  Here, the last inequality follows from our choice of~$\tilde c$ in Remark~\ref{r:ctilde} which guarantees
  \begin{equation*}
    \frac{\tilde c \lambda_N^{d/2}}{2} \leq N \leq \tilde c\lambda_N^{d/2}\,,
  \end{equation*}
  for all sufficiently large~$N$.
  This yields%
  \footnote{
    Note that in the proof of Lemma~\ref{l:H1ByL2Small}, used
    \begin{equation*}
      \sum_1^{m_0 -1} \mathcal E_\nu \theta_m
	\geq \sum_{m_0/2}^{m_0 - 1} \mathcal E_\nu \theta_m
    \end{equation*}
    and focussed on bounding the tail of the sum in order to effectively use the decay of~$h$.
    In~\eqref{e:sum1}, however, using only the tail of the sum
    does not improve our final result, and we can directly sum over the entire history.
    We only do it here because it allows us to directly use last part of the proof of Lemma~\ref{l:H1ByL2Small}.
  }
  \begin{align}\label{e:sum1}
  \sum_{m=\sfrac{m_0}{2}}^{m_0-1}\norm{P_N\phi_{m+1}}^2
    &\leq \sum_{m=0}^{m_0-1}\norm{P_N\phi_{m+1}}^2
  \\
  \nonumber
    &\leq \tilde cm_0h(m_0-1)^2\lambda_N^{(d+2\alpha+2\beta)/2}\norm{\theta_1}^2\,.
  \end{align}
  Substituting this and~\eqref{energysub1} in~\eqref{e:energy1} gives
  \begin{align}\label{e:disweaktmp2}
    \sum_{m=1}^{m_0-1}\Enu\theta_{m} \geq \frac{m_0\lambda_N}{1+m_0^2\nu\lambda_N}\Big(\frac{1}{2}- \tilde c\,h(m_0-1)^2\lambda_N^{(d+2\alpha+2\beta)/2}\Big)\norm{\theta_1}^2\,.
  \end{align}
 
  Now, the choice of $m_0$ in~\eqref{e:m0disweak} forces
   \begin{align}\label{e:m0disweaktmp}
  \tilde c\, h(m_0-1)^2\lambda_N^{(d+2\alpha+2\beta)/2}\leq \frac{1}{4}\,.
  \end{align}
  Moreover, using~\eqref{e:H2} and the fact that $\lambda_N\leq H_2(\nu)$, we see
  \begin{align}\label{e:lambdadisweaktmp}
  \lambda_N\nu m_0^2\leq 4h^{-1}\Big(\frac{1}{2\sqrt{\tilde c}}\lambda_N^{-\sfrac{(d+2\alpha+2\beta)}{4}}\Big)^2\nu \lambda_N\leq 1\,.
  \end{align}
 Substituting~\eqref{e:m0disweaktmp} and~\eqref{e:lambdadisweaktmp} in~\eqref{e:disweaktmp2}  implies~\eqref{e:tmp2}, which finishes the proof.
 \end{proof}

\section{Toral Automorphisms and the Energy Decay of Pulsed Diffusions.}\label{s:toral}

In this section we study pulsed diffusions where the underlying map~$\varphi$ is a toral automorphism, and prove Theorem~\ref{t:energydecay}.
Recall a \emph{toral automorphism} is a map of the form
\begin{equation}\label{e:toralDef}
  \varphi(x) = A x \pmod{\Z^d}\,,
\end{equation}
where $A \in \SL_d(\Z)$ is an integer valued $d \times d$ matrix with determinant~$1$.
Maps of this form are known as \emph{``cat maps''}, and one particular example is when $d = 2$ and 
\begin{equation*}
  A = \begin{pmatrix} 2 & 1\\ 1 & 1\end{pmatrix}\,.
\end{equation*}
The reason for the somewhat unusual name is that originally ``CAT'' was an abbreviation for Continuous Automorphism of the Torus.
However, it has now become tradition to demonstrate the mixing effects of this map using the image of a cat~\cite{SturmanOttinoEA06}.

\subsection{Mixing Rates of Toral Automorphisms}
It is well known that no eigenvalue of~$A$ is a root of unity, if and only if $\varphi$ is ergodic, if and only if $\varphi$ is strongly mixing (see~\cite{Katznelson71}, Page 160, problem 4.2.11 in~\cite{KatokHasselblatt95})
Our interest is in understanding the mixing rates in the sense of Definition~\ref{d:mixrate}.

\begin{proposition}\label{p:catmapExpMix}
  Let $A \in \SL_d(\Z)$ be such that:
  \begin{enumerate}[\hspace{4ex}(C1)]
    \item\mylabel{A1}{(C1)}
      No eigenvalue of~$A$ is a root of unity,
    \item\mylabel{A2}{(C2)}
      and the characteristic polynomial of~$A$ is irreducible over~$\Q$.
  \end{enumerate}
  If $\alpha, \beta > 0$ then the toral automorphism~$\varphi\colon \T^d \to \T^d$ defined by~\eqref{e:toralDef} is strongly $\alpha$, $\beta$ mixing with rate function
  \begin{equation}\label{e:toralMixRate}
    h(n) = C_{\alpha, \beta} \exp\paren[\Big]{ -\frac{n}{C_0} \paren[\Big]{\alpha \varmin \frac{\beta}{d - 1}} }\,,
  \end{equation}
  for some finite non-zero constants $C_{\alpha, \beta} = C_{\alpha, \beta}(A, \alpha, \beta)$ and $C_0 = C_0(A)$.
\end{proposition}

\begin{remark}
  Condition~\ref{A2} above is equivalent to assuming that~$A$ has no proper invariant subspaces in~$\Q^d$.
\end{remark}

For completeness, we also mention that if~$A$ satisfies Condition~\ref{A1} above, then~$A$ is also weakly $\alpha$, $\beta$ if either~$\alpha = 0$ or~$\beta = 0$ (but not both).

\begin{proposition}\label{p:catmapWeakMix}
  Let~$A \in SL_d(\Z)$ satisfy the condition~\ref{A1} in Proposition~\ref{p:catmapExpMix}.
  \begin{enumerate}
    \item
      If either $\alpha > 0$ and $\beta = 0$,  or $\alpha = 0$ and $\beta > 0$, then there exists a finite constant $C_{\alpha, \beta} = C(\alpha, \beta)$ such that~$\varphi$ is weakly $\alpha$, $\beta$ mixing with rate function
      \begin{equation}\label{e:toralMixRateWeak0}
	h(n) =
	  \begin{dcases}
	    \frac{C_{\alpha, \beta}}{\sqrt{n}}\,,
	      & \alpha \varmax \beta > \frac{d}{2}\,,
	    \\
	    C_{\alpha, \beta} \paren[\Big]{\frac{\ln n}{n}}^{1/2}\,,
	      & \alpha \varmax \beta  = \frac{d}{2}\,,
	    \\
	    \frac{C_{\alpha, \beta}}{n^{(\alpha \varmax \beta) / d }}\,,
	      & \alpha \varmax \beta < \frac{d}{2}\,.
	  \end{dcases}
      \end{equation}
      
    \item
      If further~$A$ satisfies condition~\ref{A2} in Proposition~\ref{p:catmapWeakMix}, and both $\alpha > 0$ and $\beta > 0$, then there exists a finite constant $C_{\alpha, \beta} = C(A, \alpha, \beta)$ such that~$\varphi$ is weakly $\alpha$, $\beta$ mixing with rate function
      \begin{equation}\label{e:toralMixRateWeak}
	h(n) = \frac{C_{\alpha, \beta}}{\sqrt{n}}\,.
      \end{equation}
  \end{enumerate}
\end{proposition}

When $d = 2$, Proposition~\ref{p:catmapExpMix} is well known and can be proved elementarily.
In higher dimensions, a version of Proposition~\ref{p:catmapExpMix} was proved by Lind~\cite[Theorem 6]{Lind82} using a lemma of Katznelson~\cite[Lemma 3]{Katznelson71} on Diophantine approximation.
Proposition~\ref{p:catmapExpMix} can also be deduced from the results on the algebraic structure of toral automorphisms developed in~\cite{FannjiangWoowski03}.
These arguments, however, rely on three sophisticated results from number theory: the Schmidt subspace theorem~\cite{Schmidt80}, Minkowski's theorem on linear forms~\cite[Chapter VI]{Newman72} and van der Waerdern's theorem on arithmetic progressions~\cites{Waerden27,Lukomskaya48}.
We will avoid using these results, and instead prove Proposition~\ref{p:catmapExpMix} directly using the following two algebraic lemmas.
These lemmas will be reused subsequently in the proof of sharpness of the double exponential bound~\eqref{e:dexp} in Theorem~\ref{t:energydecay}.

\begin{lemma}\label{l:diophantine}
  Suppose~$A \in \SL_d(\Z)$ satisfies the assumptions~\ref{A1} and~\ref{A2} in Proposition~\ref{p:catmapExpMix}.
  There exists a basis $\set{v_1, \dots, v_d}$ of $\C^d$ such that the following hold:
  \begin{enumerate}
    \item
      Each $v_i$ is an eigenvector of~$A$.
    \item
      If~$k \in \Z^d - 0$, and $a_i = a_i(k) \in \C$ are such that
      \begin{equation*}
	k = \sum_1^d a_i(k) v_i = \sum_1^d a_i v_i\,,
      \end{equation*}
      then we must have
      \begin{equation}\label{e:prodAi}
	\prod_{i =1}^d \abs{a_i(k)} \geq 1\,.
      \end{equation}
  \end{enumerate}
\end{lemma}

\begin{lemma}[Kronecker~\cite{Kronecker57}]\label{l:kronecker}
  Let $p$ be a monic polynomial with integer coefficients that is irreducible over~$\Q$.
  If all the roots of~$p$ are contained in the unit disk, they must be roots of unity.
\end{lemma}

The proofs of Lemma~\ref{l:diophantine} and~\ref{l:kronecker} use elementary facts about algebraic number fields, and to avoid breaking continuity, we defer the proofs to Section~\ref{s:diophantine}.
The reason these lemmas arise here is as follows.
Lemma~\ref{l:kronecker} will guarantee that~$(A^T)^{-1}$ has at least one eigenvalue, $\lambda_1$, strictly outside the unit disk.
Lemma~\ref{l:diophantine} now guarantees that all non-zero Fourier frequencies have a certain minimum component in the eigenspace of~$\lambda_1$.
This will of course dominate the long time behaviour, leading to exponential mixing of~$\varphi$ and rapid energy dissipation of the associated pulsed diffusion.

\begin{proof}[Proof of Proposition~\ref{p:catmapExpMix}]
  Let $B = (A^T)^{-1}$, and~$f \in L^2_0$.
  Observe
  \begin{equation*}
    (Uf)^\wedge(k)
      = \int_{\T^d} e^{-2\pi i k \cdot x} f(Ax) \, dx 
      = \int_{\T^d} e^{-2\pi i (B k) \cdot x} f(x) \, dx
      = \hat f( B  k )\,,
  \end{equation*}
  and hence
  \begin{equation}\label{e:UnFHatK}
    (U^n f)^\wedge(k) = \hat f( B^n k )\,,
  \end{equation}
  for all $n \geq 0$.
  Now to prove that~$\varphi$ is exponentially mixing, let $f \in \dot H^\alpha$, and $g \in \dot H^\beta$.
  Using~\eqref{e:UnFHatK} we have
  \begin{equation*}
    \ip{U^n f, g}
      = \sum_{k \in \Z^d - 0} \hat f(B^n k) \overline{\hat g(k)}
      = \sum_{k \in \Z^d - 0}
	  \frac{1}{\abs{B^n k}^\alpha \abs{k}^\beta}
	  \abs{B^n k}^\alpha \hat f(B^n k)
	  \abs{k}^\beta \overline{\hat g(k)}
  \end{equation*}
  Consequently
  \begin{equation}\label{e:unfg1}
    \abs{\ip{U^n f, g}}
      \leq
	\paren[\Big]{ \sup_{k \in \Z^d - 0} 
	  \frac{1}{\abs{B^n k}^\alpha \abs{k}^\beta}
	}
	\norm{f}_\alpha
	\norm{g}_\beta
  \end{equation}

  We now estimate the pre-factor on the right of~\eqref{e:unfg1} using Lemmas~\ref{l:diophantine} and~\ref{l:kronecker}.
  First note that~$B \in SL_d(\Z)$ also satisfies the assumptions~\ref{A1} and~\ref{A2}.
  Let $v_1$, \dots, $v_d$ be the basis given by Lemma~\ref{l:diophantine}, and $\lambda_1$, \dots, $\lambda_d$ be the corresponding eigenvalues.
  Since the characteristic polynomial of~$B$ satisfies the conditions of Lemma~\ref{l:kronecker}, we see that~$B$ has  at least one eigenvalue outside the unit disk.
  Without loss of generality we suppose~$\abs{\lambda_1} > 1$.

  By equivalence of norms on finite dimensional spaces, we know there exists $c_* > 0$ such that
  \begin{equation}\label{e:normequiv}
    \frac{1}{c_*} \abs{k'}
      \leq \paren[\Big]{ \sum \abs{a_i(k')}^2 }^{1/2}
      \leq c_* \abs{k'}\,,
      \quad\text{for all } k' \in \Z^d\,.
  \end{equation}
  Using Lemma~\ref{l:diophantine}, we note
  \begin{equation*}
    \abs{B^n k}
      = \abs[\Big]{ \sum a_i \lambda_i^n v_i }
      \geq \frac{\abs{a_1} \abs{\lambda_1}^n}{c_*}
      \geq \frac{\abs{\lambda_1}^n}{ c_* \abs{a_2} \cdots \abs{a_d} }
      \geq \frac{\abs{\lambda_1}^n}{c_*^d \abs{k}^{d-1}} \,.
  \end{equation*}
  Thus
  \begin{equation*}
    \sup_{k \in \Z^d - 0} 
      \frac{1}{\abs{B^n k}^\alpha \abs{k}^\beta}
      \leq \abs{\lambda_1}^{-n \alpha}
	\paren[\Big]{
	  \sup_{k \in \Z^d - 0} 
	  \frac{c_*^{d \alpha}}{\abs{k}^{\beta - (d - 1) \alpha} }
	}\,.
  \end{equation*}
  If $(d - 1) \alpha \leq \beta$, \eqref{e:unfg1} and the above shows that~$\varphi$ is strongly $\alpha$, $\beta$ mixing with rate function~$h(n) = C \abs{\lambda_1}^{-n \alpha}$.
  This proves~\eqref{e:toralMixRate} in the case $(d - 1) \alpha \leq \beta$.

  On the other hand, if $(d - 1) \alpha > \beta$, we let $\alpha' = \beta / (d - 1)$.
  By the previous argument we know $\varphi$ is $\alpha'$, $\beta$ mixing with rate function $h(n) = C \abs{\lambda_1}^{-n \alpha'}$.
  Since $\alpha > \alpha'$, $\norm{f}_{\alpha'} \leq \norm{f}_{\alpha}$ and  it immediately follows that~$\varphi$ is also $\alpha$, $\beta$ mixing with the same rate function.
  This proves~\eqref{e:toralMixRate} when $(d - 1) \alpha > \beta$ completing the proof.
\end{proof}

\begin{proof}[Proof of Proposition~\ref{p:catmapWeakMix}]
  The second assertion follows immediately from Proposition~\ref{p:catmapExpMix}.
  Indeed, when both $\alpha, \beta > 0$, Proposition~\ref{p:catmapExpMix} implies~$\varphi$ is strongly $\alpha$, $\beta$ mixing with rate function~$h$ given by~\eqref{e:toralMixRate}.
  Since the rate function decays exponentially, it is square summable and equation~\eqref{e:toralMixRateWeak} holds with $C_{\alpha, \beta} = (\sum_{i=1}^\infty  h(i)^2)^{1/2}$.

  To prove the first assertion, suppose first $\alpha = 0$ and $\beta > 0$.
  As before set $B = (A^T)^{-1}$, and let $f, g \in L^2_0$ and observe
  \begin{align}
    \nonumber
    \frac{1}{n} \sum_{i=0}^{n-1} \abs{\ip{U^i f, g}}^2
      &= \frac{1}{n} \sum_{i=0}^{n-1} \abs[\bigg]{
	  \sum_{k \in \Z^d - 0} \hat f( B^i k ) \overline{\hat g(k)}
      }^2
  \\
    \label{e:p42e1}
      &\leq \frac{\norm{g}_\beta^2}{n}
	\sum_{i=0}^{n-1}
	\sum_{k \in \Z^d - 0}
	  \frac{\abs{\hat f( B^i k )}^2}{\abs{k}^{2\beta} } \,.
  \end{align}

  We now split the analysis into cases.
  First suppose $\beta > d / 2$.
  By Kronecker's theorem (Lemma~\ref{l:kronecker}) we see that the matrix~$B$ can not have finite order, and hence $k$, $Bk$, $B^2 k$, \dots, $B^{n-1}k$ are all distinct.
  Thus~\eqref{e:p42e1} implies
  \begin{equation*}
    \frac{1}{n} \sum_{i=0}^{n-1} \abs{\ip{U^i f, g}}^2
      \leq \frac{\norm{g}_\beta^2}{n}
	\sum_{k \in \Z^d - 0}
	\sum_{i=0}^{n-1}
	  \frac{\abs{\hat f( B^i k )}^2}{\abs{k}^{2\beta} }
      \leq \frac{\norm{g}_\beta^2}{n}
	\sum_{k \in \Z^d - 0}
	  \frac{\norm{f}^2}{\abs{k}^{2\beta} }\,.
  \end{equation*}
  Since $\beta > d / 2$, the sum on the right is finite, showing~$\varphi$ is~$0$, $\beta$ mixing with rate function $C / n^{1/2}$ as desired.

  Suppose now $\beta < d / 2$.
  Let $m \in \N$ be a large integer that will be chosen shortly, and split the above sum as
  \begin{align}
    \label{e:tmp410}
    \frac{1}{n} \sum_{i=0}^{n-1} \abs{\ip{U^i f, g}}^2
      &\leq
	\frac{\norm{g}_\beta^2}{n}
	\paren[\Big]{ 
	  \sum_{0 < \abs{k} \leq m }
	    \sum_{i=0}^{n-1}
	    \frac{\abs{\hat f( B^i k )}^2}{\abs{k}^{2\beta} }
	  +
	  \sum_{i=0}^{n-1}
	    \sum_{\abs{k} > m }
	      \frac{\abs{\hat f( B^i k )}^2}{\abs{k}^{2\beta} }
	}
    \\
      \label{e:tmp411}
      &\leq
	\norm{f}^2 \norm{g}_\beta^2
	\brak[\Big]{ 
	  \paren[\Big]{\frac{1}{n}
	    \sum_{0 < \abs{k} \leq m }
	    \frac{1}{\abs{k}^{2\beta} }}
	  +
	  \frac{1}{m^{2\beta}}
	}
    \\
      \nonumber
      &\leq \norm{f}^2 \norm{g}_\beta^2 \paren[\Big]{
	\frac{C m^{d - 2\beta}}{n} + \frac{1}{m^{2\beta}}
      }\,,
  \end{align}
  for some (explicit) constant~$C = C(d)$, independent of~$n$.
  (Note, we again used the fact that $k$, $B k$, $B^2 k$, \dots, are all distinct when computing the first sum on the right of~\eqref{e:tmp410} to obtain~\eqref{e:tmp411}.)
  We now choose~$m = C n^{1/d}$ in order to minimize the right hand side.
  This implies
  \begin{equation*}
    \frac{1}{n} \sum_{i=0}^{n-1} \abs{\ip{U^i f, g}}^2
      \leq \frac{C \norm{f}^2 \norm{g}_\beta^2}{n^{2\beta/d}}
  \end{equation*}
  proving~\eqref{e:toralMixRateWeak0} when $\beta < d / 2$.

  Finally, when $\beta = d/2$ we repeat the same argument above to obtain~\eqref{e:tmp411}.
  When summed~\eqref{e:tmp411} now yields
  \begin{equation}\label{e:tmp413}
    \frac{1}{n} \sum_{i=0}^{n-1} \abs{\ip{U^i f, g}}^2
      \leq \norm{f}^2 \norm{g}_\beta^2 \paren[\Big]{
	\frac{C \ln m}{n} + \frac{1}{m^d}
      }\,,
  \end{equation}
  and choosing~$m = n^{1/d}$ yields~\eqref{e:toralMixRateWeak0} as desired.

  We have now proved~\eqref{e:toralMixRateWeak0} when $\alpha = 0$ and $\beta > 0$.
  For the case~$\alpha > 0$ and $\beta = 0$, note that $\ip{U^i f, g} = \ip{f, U^{-i} g}$.
  Thus replacing the matrix~$A$ with~$A^{-1}$ reduces the case when $\alpha > 0, \beta = 0$ to the case when $\alpha = 0$, $\beta > 0$.
  This finishes the proof.
\end{proof}

\subsection{Energy Decay, and the proof of Theorem~\ref{t:energydecay}}\label{s:energydecay}

We now turn our attention to studying the energy decay of pulsed diffusions.
Our first result shows that if a toral automorphism satisfies conditions~\ref{A1} and~\ref{A2} in Proposition~\ref{p:catmapExpMix}, then the energy of the associated pulsed diffusion decays double exponentially.
This will prove sharpness of the lower bound~\eqref{e:dexp} in Theorem~\ref{t:energydecay}.
Following this we will prove lower bound~\eqref{e:dexp} itself using a convexity argument.

\begin{proposition}\label{p:toralEnergy}
  Suppose~$A \in \SL_d(\Z)$ satisfies the assumptions~\ref{A1} and~\ref{A2} in Proposition~\ref{p:catmapExpMix}.
  Let~$\varphi$ be the associated toral automorphism defined in~\eqref{e:toralDef}, and~$\theta_n$ be the pulsed diffusion defined by~\eqref{e:pulseddiff2}.
  Then there exist constants $c > 0$ and $\gamma > 1$ such that
  \begin{equation}\label{e:dexp2}
    \norm{\theta_n} \leq \exp\paren[\Big]{ - \frac{\nu \gamma^n}{c} }
  \end{equation}
\end{proposition}
\begin{remark}
  In the proof of Proposition~\ref{p:toralEnergy} we will see that the constant~$\gamma$ can be chosen to be
  \begin{equation*}
    \gamma = \prod_{i = 1}^d \paren{\abs{\lambda_i} \varmax 1 }^{2/d}
  \end{equation*}
  where $\lambda_1$, \dots, $\lambda_d$ are the eigenvalues of~$A$.
\end{remark}
\begin{proof}
  Using~\eqref{e:UnFHatK} we see
  \begin{equation*}
    \hat \theta_{n+1}(k) = e^{-\nu \abs{k}^2} \hat \theta_n( B k )\,.
  \end{equation*}
  Setting $A_* = A^T$, iterating the above, squaring and summing in~$k$ gives
  \begin{equation}\label{e:thetan1}
    \norm{\theta_n}^2
      = \sum_{k \in \Z^d - 0}
	\exp\paren[\Big]{ - 2\nu \sum_{j=1}^n \abs{A_*^j k}^2 } \abs{\hat \theta_0(k)}^2\,.
  \end{equation}

  Observe that the matrix~$A_*$ also satisfies the conditions~\ref{A1} and~\ref{A2} in Proposition~\ref{p:catmapExpMix}.
  Let~$v_1$, \dots, $v_d$ be the basis of $\C^d$ given by Lemma~\ref{l:diophantine}, and~$\lambda_1$, \dots, $\lambda_d$ be the corresponding eigenvalues.
  Now~\eqref{e:thetan1} implies
  \begin{align}
    \nonumber
    \norm{\theta_n}^2
      &\leq \sum_{k \in \Z^d - 0}
	\exp\paren[\Big]{ - \frac{2\nu}{c_*^2} \sum_{j=1}^n \sum_{i=1}^d \abs{a_i}^2 \abs{\lambda_i}^{2j} } \abs{\hat \theta_0(k)}^2
    \\
    \nonumber
      &= \sum_{k \in \Z^d - 0}
	\exp\paren[\Big]{ - \frac{2\nu}{c_*^2} \sum_{i=1}^d \abs{a_i}^2 \paren[\Big]{\frac{\abs{\lambda_i}^{2(n+1)} - \abs{\lambda_i}^2 }{\abs{\lambda_i}^2 - 1}} }  \abs{\hat \theta_0(k)}^2
    \\
    \label{e:thetan2}
      &\leq \norm{\theta_0}^2 \sup_{k \in \Z^d - 0}
	\exp\paren[\Big]{ - \frac{2\nu}{c_*^2} \sum_{i=1}^d \abs{a_i}^2 \paren[\Big]{\frac{\abs{\lambda_i}^{2(n+1)} - \abs{\lambda_i}^2 }{\abs{\lambda_i}^2 - 1}} }\,.
  \end{align}
  where $c_*$ is the constant in~\eqref{e:normequiv}.
  
  We will now show that the last term decays double exponentially in~$n$.
  Indeed, the inequality of the means implies
  \begin{align}
    \nonumber
    \sum_{i=1}^d
      \abs{a_i}^2 \paren[\Big]{\frac{\abs{\lambda_i}^{2(n+1)} - \abs{\lambda_i}^2 }{\abs{\lambda_i}^2 - 1} }
      &\geq
	d \paren[\Big]{ \prod_{i=1}^d \abs{a_i}^2
	    \paren[\Big]{\frac{\abs{\lambda_i}^{2(n+1)} - \abs{\lambda_i}^2 }{\abs{\lambda_i}^2 - 1} } }^{1/d}
    \\
    \nonumber
      &= d \paren[\Big]{ \prod_{i=1}^d \abs{a_i}^2 }^{1/d}
	  \paren[\Big]{ \prod_{i=1}^d
	    \paren[\Big]{\frac{\abs{\lambda_i}^{2(n+1)} - \abs{\lambda_i}^2 }{\abs{\lambda_i}^2 - 1} } }^{1/d}
    \\
    \label{e:amgm1}
      &\geq d \paren[\Big]{ \prod_{i=1}^d
	    \paren[\Big]{\frac{\abs{\lambda_i}^{2(n+1)} - \abs{\lambda_i}^2 }{\abs{\lambda_i}^2 - 1} } }^{1/d}\,,
  \end{align}
  where the last inequality followed from Lemma~\ref{l:diophantine}.
  As in the proof of Proposition~\ref{p:catmapExpMix}, Lemma~\ref{l:kronecker} guarantees that $\max_i \abs{\lambda_i} > 1$.
  The right hand side of~\eqref{e:amgm1} is of order $\prod_i (\abs{\lambda_i} \varmax 1)^{2n/d}$ and substituting this in~\eqref{e:thetan2} gives~\eqref{e:dexp2} as desired.
\end{proof}

We now prove Theorem~\ref{t:energydecay}.
\begin{proof}
  Proposition~\ref{p:toralEnergy} immediately shows that the double exponential upper bound equation~\eqref{e:dexpUpper} is achieved for the desired class of toral automorphisms.
  Thus it only remains to prove the double exponential lower bound~\eqref{e:dexp}.
  For this, observe
  \begin{equation*}
    \ln \norm{\theta_{n+1}}^2 - \ln \norm{\theta_n}^2
      = \ln \paren[\Big]{ \frac{\norm{\theta_{n+1}}^2}{\norm{\theta_n}^2}}
      = \ln \paren[\Big]{ \frac{\norm{\theta_{n+1}}^2}{\norm{U\theta_n}^2}}
      = \ln \paren[\Big]{\frac{\sum_i e^{-2\nu\lambda_i}\abs{\ip{U\theta_n,e_i}}^2}{\sum_i\abs{\ip{U\theta_n,e_i}}^2}}\,,
  \end{equation*}
  where we recall that $\lambda_i$ are the eigenvalues of the Laplacian, and $e_i$'s are the corresponding eigenfunctions.
  Using concavity of the logarithm and Jensen's inequality to bound the last term on the right we obtain
  \begin{align}
    \nonumber
    \ln \norm{\theta_{n+1}}^2 - \ln \norm{\theta_n}^2
      &\geq \frac{-2\nu\sum_i\lambda_i\abs{\ip{U\theta_n,e_i}}^2}{\sum_i\abs{\ip{U\theta_n,e_i}}^2}
      = -2\nu \frac{\norm{U \theta_n}_1^2}{\norm{U\theta_n}^2}
      \\
      \label{e:lnNormThN1}
      &\geq
	 -2\nu \norm{\grad \varphi}_{L^\infty}^2
	 \frac{\norm{\theta_n}_1^2}{\norm{\theta_n}^2}\,.
  \end{align}

  We now claim
  \begin{equation}\label{e:rbound1}
    \frac{\norm{\theta_n}_1^2}{\norm{\theta_n}^2}
      \leq \norm{\grad \varphi}_{L^\infty}^{2n}
	\frac{\norm{\theta_0}_1^2}{\norm{\theta_0}^2}\,.
  \end{equation}
  Note that substituting~\eqref{e:rbound1} in~\eqref{e:lnNormThN1} and summing in~$n$ immediately implies~\eqref{e:dexp}.
  Thus to finish the proof we only need to prove~\eqref{e:rbound1}.

  For this we observe
  \begin{align*}
    \MoveEqLeft
    \frac{\norm{\theta_{n+1}}_1^2 }{\norm{\theta_{n+1}}^2} 
      - \frac{\norm{U \theta_{n}}_1^2 }{\norm{U \theta_n}^2} 
      = \frac{\norm{\theta_{n+1}}_1^2 \norm{U\theta_n}^2 - \norm{\theta_{n+1}}^2 \norm{U \theta_n}_1^2}{\norm{\theta_n}^2 \norm{U \theta_n}^2} 
    \\
      &= \frac{1}{\norm{\theta_n}^{2} \norm{U \theta_n}^{2}}
	\paren[\Big]{
	  \sum_{i,j} e^{-2\nu\lambda_i}(\lambda_i-\lambda_j)\abs{\ip{U\theta_n,e_i}}^2\abs{\ip{U\theta_n,e_j}}^2
	}
    \\
    &=\frac{1}{\norm{\theta_n}^{2} \norm{U \theta_n}^{2}}
	\begin{multlined}[t]
      \Bigl(
	  \sum_{i<j}e^{-2\nu\lambda_i}(\lambda_i-\lambda_j)\abs{\ip{U\theta_n,e_i}}^2\abs{\ip{U\theta_n,e_j}}^2
	\\
       + \sum_{i>j}e^{-2\nu\lambda_i}(\lambda_i-\lambda_j)\abs{\ip{U\theta_n,e_i}}^2\abs{\ip{U\theta_n,e_j}}^2
      \Bigr)
	\end{multlined}
    \\
    &\leq\frac{1}{\norm{\theta_n}^{2} \norm{U \theta_n}^{2}}
	\begin{multlined}[t]
      \Bigl(
	  \sum_{i<j}e^{-2\nu\lambda_i}(\lambda_i-\lambda_j)\abs{\ip{U\theta_n,e_i}}^2\abs{\ip{U\theta_n,e_j}}^2
	\\
       + \sum_{i>j}e^{-2\nu\lambda_j}(\lambda_i-\lambda_j)\abs{\ip{U\theta_n,e_i}}^2\abs{\ip{U\theta_n,e_j}}^2
      \Bigr)
	\end{multlined}
    \\
    &=0\,.
  \end{align*}
  Thus
  \begin{equation*}
    \frac{\norm{\theta_{n+1}}_1^2 }{\norm{\theta_{n+1}}^2} 
      \leq \frac{\norm{U \theta_{n}}_1^2 }{\norm{U \theta_n}^2} 
      = \frac{\norm{U \theta_{n}}_1^2 }{\norm{\theta_n}^2} 
      \leq \norm{\grad{\varphi}}_{L^\infty}^2 \frac{\norm{\theta_{n}}_1^2 }{\norm{\theta_n}^2} \,,
  \end{equation*}
  and iterating yields~\eqref{e:rbound1}.
  This finishes the proof.
\end{proof}

\subsection{Diophantine Approximation and Kronecker's Theorem}\label{s:diophantine}

We now prove Lemmas~\ref{l:diophantine} and~\ref{l:kronecker}.
The proofs rely on standard facts on algebraic number fields, and we refer the reader to the books~\cites{Marcus77} and~\cite{Ribenboim01} for a comprehensive treatment.

Before beginning the proof, we remark that a weaker version of Lemma~\ref{l:diophantine} follows directly from the Schmidt subspace~\cite{Schmidt80}.
Explicitly, the Schmidt subspace theorem guarantees that for any $\epsilon > 0$ we have
\begin{equation*}
  \abs[\Big]{ \prod_{i=1}^d a_i(k) } \geq \frac{1}{\abs{k}^\epsilon}\,,
\end{equation*}
at all integer points $k \in \Z^d$, \emph{except} on finitely many proper rational subspaces.
To use the Schmidt subspace theorem in our context we would need to handle the exceptional subspaces.
The approach taken by Fannjiang et.\ al. in~\cite{FannjiangWoowski03} is to use van der Waerdern's theorem on arithmetic progressions~\cites{Waerden27,Lukomskaya48} to  construct an equivalent minimization problem whose minimizer is guaranteed to lie outside the exceptional subspaces.
In our specific context we can directly prove the stronger bound~\eqref{e:prodAi}, and avoid using the Schmidt subspace theorem entirely.

\begin{proof}[Proof of Lemma~\ref{l:diophantine}]
  Let $p$ be the characteristic polynomial of~$A$, and $\lambda_1$, \dots, $\lambda_d$ be the roots of $p$.
  Let $F = \Q(\lambda_1, \dots, \lambda_d)$ and $\mathcal G = \operatorname{Gal}(F / \Q)$ denote the Galois group.
  Let $G_1 \subseteq \mathcal G$ be the group of field automorphisms that fix $\lambda_i$, and let
  \begin{equation*}
    F_1 = \set{x \in F \st \sigma(x) = x \ \forall \sigma \in G_1}\,,
  \end{equation*}
  be the fixed field of $G_1$.
  Since $\det( A - \lambda_1 I) = 0$, there must exist $v_1$ in the $F_1$ vector space $F_1^d$ such that $A v_1 = \lambda_1 v_1$.
  For $i \neq 1$, let $\tau_i \in \mathcal G$ be any element such that $\tau_i(\lambda_1) = \lambda_i$.
  (Since $p$ is irreducible over~$\Q$, the Galois group $\mathcal G$ acts transitively on the roots $\lambda_1, \dots, \lambda_d$, and hence such an element~$\tau_i$ must exist.)
  Now we define
  \begin{equation*}
    v_i \defeq \tau_i(v_1)\,.
  \end{equation*}

  We now view each $v_i$ as an element of $\C^d$, we let $V \in \GL_d(\C)$ be the matrix with columns $v_1$, \dots, $v_d$.
  Dividing each $v_i$ by a large integer if necessary, we may assume that each entry of $V^{-1}$ is an algebraic integer.
  We claim that $v_1$, \dots, $v_d$ is the desired basis.

  To see this we first note that the basis $\set{v_1, \dots, v_d}$ has the following property:
    if $\sigma \in \mathcal G$ is such that $\sigma(\lambda_i) = \lambda_j$, then $\sigma(v_i) = v_j$.
  Indeed, note that $\tau_j^{-1} \sigma \tau_i(\lambda_1) = \lambda_1$, and hence $\tau_j^{-1} \sigma \tau_i \in G_1$.
  Since all coordinates of the vector $v_1$ are in $F_1$, the fixed field of $G_1$, this must mean that $\tau_j^{-1} \sigma \tau_i (v_1) = v_1$.
  This implies $\sigma(v_i) = v_j$ as claimed.

  Now we show that the basis $\set{v_1, \dots, v_d}$ has the second property stated in Lemma~\ref{l:diophantine}.
  Let $k \in \Z^d - \set{0}$, choose $a_i = a_i(k) \in \C$ such that $k = \sum a_i v_i$, and define
  \begin{equation*}
    p_* \defeq \prod_{\sigma \in \mathcal G} \sigma(a_1)\,.
  \end{equation*}
  Note that if $\sigma(\lambda_i) = \lambda_j$, then $\sigma(v_i) = v_j$ and hence $\sigma(a_i) = a_j$.
  Consequently,
  \begin{equation*}
    p_* 
      = \prod_{\sigma \in G_1} \prod_{i = 1}^d \tau_i \sigma (a_1)
      = \paren[\Big]{ \prod_{i=1}^d a_i  }^m\,,
  \end{equation*}
  where $m = \abs{G_1}$.
  Thus $p_*$ is in the fixed field of $\mathcal G$, and hence must be rational.

  Further, since $a_i = (V^{-1} k) \cdot e_i$, each $a_i$ must also be an algebraic integer.
  This forces $p_*$ to be a rational algebraic integer, and hence an integer.
  By transitivity of the Galois group we see that if  $a_i = 0$ for some~$i$, then we must have $a_j = 0$ for all $j$.
  Thus $p_*$ must be a non-zero, and hence $\abs{p_*} \geq 1$.
  This proves~\eqref{e:prodAi} as desired.
\end{proof}

Lemma~\ref{l:kronecker} is due to Kronecker~\cite{Kronecker57}.
This result was improved by Stewart~\cite{Stewart78} and Dobrowolski~\cite{Dobrowolski79}.
More generally Lehmer's conjecture~\cite{Lehmer33} asserts that if $\lambda_1$, \dots, $\lambda_d$ are the roots of~$p$ and the product $\prod (1 \varmax \abs{\lambda_i})$ is smaller than an absolute constant $\mu$ (widely believed to be approximately $1.176\dots$), then each~$\lambda_i$ is a root of unity.
For our purposes, however, Kronecker's original result will suffice.
Since the proof is short and elementary, we present it below.

\begin{proof}[Proof of Lemma~\ref{l:kronecker}]
  Let $\lambda_1$, \dots, $\lambda_d$ be the roots of~$p$.
  For any $n \in \N$, let $p_n$ be the minimal monic polynomial satisfied by~$\lambda_1^n$.
  Since the Galois conjugates of $\lambda_1^n$ are precisely $\lambda_2^n$, \dots, $\lambda_d^n$, the coefficients of $p_n$ are symmetric functions of $\lambda_1^n$, \dots, $\lambda_d^n$.
  By assumption $\abs{\lambda_i} \leq 1$, which implies~$\abs{\lambda_i^n} \leq 1$, which in turn implies that the coefficients of $p_n$ are uniformly bounded as functions of $n$.
  There are only finitely many polynomials with degree at most $d$, and uniformly bounded integer coefficients.
  Thus for some infinite set $A \subseteq \N$, we must have $p_m = p_n$ for all $m, n \in A$.
  This forces the existence of one $i \in \set{1, \dots, d}$ and distinct $m, n \in \N$ such that $\lambda_i^m = \lambda_i^n$.
  Hence $\lambda_i$ is a root of unity.
  Since $\lambda_1$, \dots, $\lambda_d$ are all Galois conjugates, they must all be roots of unity.
\end{proof}
\section{Dissipation Enhancement for the advection diffusion equation.}\label{s:continuous}
We now prove Theorems~\ref{t:disStrongCts} and~\ref{t:disWeakCts}, bounding the dissipation time in the continuous time setting.
The main idea is similar to the discrete time case.
However, in the continuous time setting the approximation of the diffusive system by the underlying dynamical system is not as good as in the discrete time setting.
This is the reason why the estimates in Theorems~\ref{t:disStrongCts} and~\ref{t:disWeakCts} are not as strong as those in Theorems~\ref{t:disStrong} and~\ref{t:disWeak}.

\subsection{The Strongly Mixing Case} 

As in Section~\ref{s:ad}, let $\theta_{s,0}\in L_0^2(M)$, let $\theta_s(t)$ be the solution of~\eqref{e:ad}.
By the energy inequality~\eqref{e:energyCts} we know
\begin{equation*}
  \norm{\theta_s(t)}^2
    = \norm{\theta_s(s)}^2
      \exp\paren[\Big]{-2\nu \int_s^t
	\frac{\norm{\theta_s(r)}_1^2}{\norm{\theta_s(r)}^2} \, dr }\,.
\end{equation*}
Thus,~$\norm{\theta_s(t)}$ decays rapidly when the ratio $\norm{\theta_s(t)}_1/\norm{\theta_s(t)}$ remains large.
Precisely, if for some $c_0>0$, we have
\begin{align*}
\norm{\theta_s(t)}_{1}^2 \geq c_0\norm{\theta_s(t)}^2\,, \quad \text{ for all } s\leq t \leq t_0\,,
\end{align*}
then
\begin{align}\label{e:H1Large}
\norm{\theta_s(t)}^2\leq e^{-2\nu c_0(t-s)}\norm{\theta_{s,0}}^2\,, \quad \text{ for all } s\leq t \leq t_0\,.
\end{align}
As in the proof of Theorems~\ref{t:disStrong} and~\ref{t:disWeak}, we will show that if the ratio~$\norm{\theta_{s,0}}_1/\norm{\theta_{s,0}}$ is small, then the mixing properties of~$u$ will guarantee that for some later time $t_0 > s$, $\norm{\theta_s(t_0)}$ becomes sufficiently small.
This is the content of the following lemma.

\begin{lemma}\label{l:H1ByL2SmallCts}
  Choose $\lambda_N$ to be the largest eigenvalue satisfying $\lambda_N\leq H_3(\nu)$ where $H_3(\nu)$ is defined in~\eqref{e:H3}. If
  \begin{equation}\label{e:H1ByL2SmallCts}
  \norm{\theta_{s,0}}_1^2
      < \lambda_N \norm{\theta_{s,0}}^2\,,
  \end{equation}
  then we have
  \begin{equation}\label{e:thetaT0Cts}
    \norm{ \theta_s(t_0) }^2 \leq \exp\paren[\Big]{- \frac{\nu H_3(\nu) (t_0-s)}{8} } \norm{\theta_{s,0}}^2\,.
  \end{equation}
  at a time $t_0$ given by
  \begin{equation*}
    t_0 \defeq s + 2h^{-1}\Big(\frac{\lambda_N^{-\sfrac{(\alpha+\beta)}{2}}}{2}\Big) \,.
  \end{equation*}
\end{lemma}

Momentarily postponing the proof of Lemma~\ref{l:H1ByL2SmallCts}, we prove Theorem~\ref{t:disStrongCts}.

\begin{proof}[Proof of Theorem~\ref{t:disStrongCts}]
  Choosing $c_0=\lambda_N$ and repeatedly applying the inequality~\eqref{e:H1Large} and Lemma~\ref{l:H1ByL2SmallCts}, we obtain an increasing sequence of times $(t'_k)$, such that
\begin{align*}
\norm{\theta_s(t_k')}^2 \leq \exp \Big(-\frac{\nu H_3(\nu) (t_k'-s)}{8}\Big)\norm{\theta_{s,0}}^2\,, \text{ and } t_{k+1}'-t_k'\leq t_0\,.
\end{align*} 
This immediately implies
\begin{align}\label{e:tmpCts5}
\tau_d \leq \frac{16}{\nu H_3(\nu)}+(t_0 -s )\,.
\end{align}
By choice of $\lambda_N$ and $t_0$, we know that $t_0 - s\leq \sfrac{1}{(\nu \lambda_N)} \leq \sfrac{2}{(\nu H_3(\nu))}$ for $\nu$ sufficiently small.
The last inequality followed from Weyl's lemma as in the proof Theorem~\ref{t:disStrong} (equation~\eqref{e:lambdahnu}).
This proves~\eqref{e:taudStrongCts} as desired.
\end{proof}

We now compute~$H_3$ explicitly when the mixing rate function decays exponentially, or polynomially.
\begin{proof}[Proof of Corollary~\ref{c:SMixExp}]
  Suppose first the mixing rate function~$h$ satisfies the power law~\eqref{e:powerh}.
  In this case the inverse is given by $h^{-1}(t)=(\sfrac{c}{t})^{1/p}$.
  Thus, by definition of~$H_3$ (in~\eqref{e:H3}), we have
\begin{align*}
\exp\paren[\Big]{2^{(2p+1)/p}c^{1/p}\norm{\nabla u}_{L^\infty}H_3(\nu)^{\frac{\alpha+\beta}{2p}}} = \frac{(2c)^{1/p}\norm{\nabla u}_{L^\infty}^2}{2\nu} H_3(\nu)^{\frac{\alpha+\beta-2p}{2p}}\,.
\end{align*}
Since $H_3(\nu) \to \infty$ as $\nu \to 0$, the above forces
\begin{align*}
H_3(\nu) \approx C\abs{\ln \nu}^{\frac{2p}{\alpha+\beta}}\,,
\end{align*}
asymptotically as $\nu \to 0$, for some constant $C=C(c,p,\alpha,\beta,\norm{\nabla u}_{L^\infty})$.
Using this in~\eqref{e:taudStrongCts} yields~\eqref{e:SMixExpPower} as desired.

Suppose now the rate function~$h$ is the exponential~\eqref{e:exp}.
Then we see $h^{-1}(t)= \sfrac{(\ln c_1 - \ln t)}{c_2}$.
By the definition of ~$H_3$ in~\eqref{e:H3}, we have
\begin{multline*}
H_3(\nu)\exp\paren[\Big]{ \frac{4\norm{\nabla u}_{L^\infty}}{c_2}
  \paren[\Big]{
    \ln (2 c_1)
    +\frac{\alpha+\beta}{2} \ln H_3(\nu)
  }}\\
= \frac{\norm{\nabla u}_{L^\infty}^2}{2\nu c_2}\paren[\Big]{\ln (2 c_1) + \frac{\alpha+\beta}{2}\ln H_3(\nu)}\,.
\end{multline*}
Taking the logarithm of both sides shows
\begin{align*}
  H_3(\nu)
    = O\paren[\Big]{
	\frac{1}{\nu^{1 - \delta}}
      }
\,,
\end{align*}
asymptotically as $\nu \to 0$, where $\delta$ is defined in~\eqref{e:SMixExpExp}.
Substituting this in~\eqref{e:taudStrongCts} yields~\eqref{e:SMixExpExp} as desired.
\end{proof}

It remains to prove Lemma~\ref{l:H1ByL2SmallCts}.
For this we will need a standard result estimating the difference between~$\theta$ and solutions to the inviscid transport equation.
\begin{lemma}\label{l:l2diffcts}
Let $\phi_s$, defined by 
\begin{align*}
\phi_s=\theta_{s,0}\circ \varphi_{s,t}\,,
\end{align*}
be the evolution of $\theta_{s,0}$ under the dynamical system generated by $\varphi_{s,t}$. If $\theta_{s,0}\in \dot H^1(M)$, then for all $t\geq s$, we have 
\begin{align}\label{e:thetaphidiff}
\norm{\theta_s(t)-\phi_s(t)}^2 \leq  \frac{\nu}{2 \norm{\grad u}_{L^\infty}} \, \exp\paren[\big]{ {2\norm{\nabla u}_{L^\infty}(t-s) }} \, \norm{\theta_{s,0}}_{1}^2\,.
\end{align}
\end{lemma}
\begin{proof}
Let $w(t) = \theta_s(t)-\phi_s(t)$.
Note $w(s) = 0$, and for $t \geq s$ we have
\begin{align*}
\partial_t w +u \cdot \nabla  w -\nu \Delta  w  =\nu \lap \phi_s\,.
\end{align*}
Multiplying both sides by $w$ and integrating over~$M$ gives
\begin{align*}
\frac{1}{2}\partial_t \norm{w}^2+\nu\norm{w}_{1}^2
  =\nu \int_M w \Delta \phi_s \,dx
\leq \frac{\nu}{2}\norm{w}_{1}^2+\frac{\nu}{2}\norm{\phi_s}_{1}^2\,,
\end{align*}
and hence
\begin{align}\label{e:tmpcts1}
\partial_t \norm{w}^2 \leq \nu\norm{\phi_s}_{1}^2\,.
\end{align}
Since $\phi_s(t) = \theta_{s, 0} \circ \varphi_{s, t}$ we know
\begin{align*}
\norm{\phi_s(t)}_{1}
  \leq
    \exp\paren[\big]{{\norm{\nabla u}_{L^\infty}(t-s)}}
    \norm{\theta_{s,0}}_{1}\,.
\end{align*}
Substituting this into~\eqref{e:tmpcts1} and integrating in time yields~\eqref{e:thetaphidiff} as claimed.
\end{proof}

We can now prove Lemma~\ref{l:H1ByL2SmallCts}.

\begin{proof}[Proof of Lemma \ref{l:H1ByL2SmallCts}]
Integrating the energy equality~\eqref{e:energyCts} gives
\begin{align}\label{e:energyCts2}
\norm{\theta_s(t_0)}^2 = \norm{\theta_{s,0}}^2-2\nu\int_s^{t_0}\norm{\theta_s(r)}_1^2\,dr\,.
\end{align}
We claim that our choice of $\lambda_N$ and $t_0$ will guarantee
\begin{align}\label{e:lowerbound}
\int_s^{t_0}\norm{\theta_s(r)}_1^2\,dr \geq \frac{\lambda_N(t_0-s)\norm{\theta_{s,0}}^2}{8}\,.
\end{align}
This immediately yields~\eqref{e:thetaT0Cts} since when $\nu$ is small enough, we have~$\frac{1}{2}H_3(\nu)\leq \lambda_N\leq H_3(\nu)$. And so to finish the proof we only have to prove~\eqref{e:lowerbound}.

Note first
\begin{align}\label{e:tmpcts2}
\nonumber
\int_s^{t_0}\norm{\theta_s(r)}_1^2\,dr & \geq  \lambda_N\int_{\frac{t_0+s}{2}}^{t_0}\norm{(I-P_N)\theta_s(r)}^2\,dr\\
\nonumber
& \geq \frac{\lambda_N}{2}\int_{\frac{t_0+s}{2}}^{t_0}\norm{(I-P_N)\phi_s(r)}^2\,dr\\
\nonumber
&\qquad -\lambda_N\int_{\frac{t_0+s}{2}}^{t_0}\norm{(I-P_N)\paren[\big]{\theta_s(r)-\phi_s(r)}}^2\,dr\\
&\geq \frac{\lambda_N(t_0-s)}{4}\norm{\theta_{s,0}}^2-\frac{\lambda_N}{2}\int_{\frac{t_0+s}{2}}^{t_0}\norm{P_N \phi_s(r)}^2\,dr\\
\nonumber
&\qquad -\lambda_N\int_{\frac{t_0+s}{2}}^{t_0}\norm{\theta_s(r)-\phi_s(r)}^2\,dr\,.
\end{align}
We will now bound the last two terms in~\eqref{e:tmpcts2}.
For the second term, note the strong mixing assumption~\eqref{e:hmixrateStrongCts} gives
\begin{align}
  \nonumber
\int_{\frac{t_0+s}{2}}^{t_0}\norm{P_N\phi_s(r)}^2\,dr \leq \lambda_N^\beta\int_{\frac{t_0+s}{2}}^{t_0}\norm{\phi_s(r)}_{-\beta}^2\,dr\leq \lambda_N^\beta \int_{\frac{t_0+s}{2}}^{t_0}h(r-s)^2 \norm{\theta_{s,0}}_{\alpha}^2\,dr\\
  \label{e:PnBound}
 \leq \frac{t_0-s}{2}\lambda_N^{\beta}h\Big(\frac{t_0-s}{2}\Big)^2\norm{\theta_{s,0}}_\alpha^2 \leq \frac{t_0-s}{2}\lambda_N^{\beta}h\Big(\frac{t_0-s}{2}\Big)^2 \norm{\theta_{s,0}}^{2-2\alpha}\norm{\theta_{s,0}}_1^{2\alpha}\,.
\end{align}
Using the assumption \eqref{e:H1ByL2SmallCts}, we obtain
\begin{align}\label{e:tmpCts3}
\int_{\frac{t_0+s}{2}}^{t_0}\norm{P_N\phi_s(r)}^2\,dr& \leq \frac{t_0-s}{2}\lambda_N^{\alpha+\beta}h\Big(\frac{t_0-s}{2}\Big)^2 \norm{\theta_{s,0}}^2\,.
\end{align}

Now we bound the last term in~\eqref{e:tmpcts2}.
Using Lemma~\ref{l:l2diffcts} we obtain
\begin{align}\label{e:tmpCts4}
  \nonumber
  \int_{\frac{t_0+s}{2}}^{t_0}\norm{\theta_s(r)-\phi_s(r)}^2\,dr &
    \leq \frac{\nu}{4 \norm{\grad u}_{L^\infty}^2} e^{2\norm{\nabla u}_{L^\infty}(t_0-s)}\norm{\theta_{s,0}}_1^2\\
  & \leq \frac{\nu \lambda_N}{4 \norm{\grad u}_{L^\infty}^2} e^{2\norm{\nabla u}_{L^\infty}(t_0-s)}\norm{\theta_{s,0}}^2\,.
\end{align}

Substituting~\eqref{e:tmpCts3} and~\eqref{e:tmpCts4} into~\eqref{e:tmpcts2} gives
\begin{equation*}
\int_s^{t_0}\norm{\theta_s(r)}_1^2\,dr
\geq \lambda_N(t_0-s)\norm{\theta_{s,0}}^2\Big(\frac{1}{4}-\frac{\lambda_N^{\alpha+\beta}}{4}h\Big(\frac{t_0-s}{2}\Big)^2-\frac{\nu\lambda_N e^{2\norm{\nabla u}_{L^\infty}(t_0-s)}}{4 \norm{\grad u}_{L^\infty}^2 (t_0-s)}\Big)
\end{equation*}
By our choice of~$\lambda_N$ and $t_0$,  we have
\begin{align*}
\frac{\lambda_N^{\alpha+\beta}}{4}h\Big(\frac{t_0-s}{2}\Big)^2\leq \frac{1}{16}\,,
\qquad\text{and}\qquad
\frac{\nu\lambda_N e^{2\norm{\nabla u}_{L^\infty}(t_0-s)}}{4 \norm{\grad u}_{L^\infty}^2 (t_0-s)}\leq \frac{1}{16}\,,
\end{align*}
from which~\eqref{e:lowerbound} follows.
This finishes the proof of Lemma~\ref{l:H1ByL2SmallCts}.
\end{proof}

\subsection{The Weakly Mixing Case.}\label{s:weakMixCts}

We now turn our attention to Theorem~\ref{t:disWeakCts}.
The proof is similar to the proof of Theorem~\ref{t:disStrongCts}.
The main difference is that the analog of Lemma~\ref{l:H1ByL2SmallCts} is weaker.
\begin{lemma}\label{l:ctsWeakH1small}
Let $\lambda_N$ to be the largest eigenvalue of $-\lap$ such that $\lambda_N\leq H_4(\nu)$, where we recall that the function $H_4$ is defined in~\eqref{e:H4}. If
\begin{align}
\norm{\theta_{s,0}}_1^2< \lambda_N \norm{\theta_{s,0}}^2\,,
\end{align}
then we have
\begin{align}
\norm{\theta_s(t_0)}^2\leq \exp\paren[\Big]{-\frac{\nu H_4(\nu)(t_0-s)}{8}}\norm{\theta_{s,0}}^2\,,
\end{align}
at a time $t_0$ given by
\begin{equation*}
  t_0=s + 2h^{-1} \paren[\Big]{ \frac{1}{2\sqrt {\tilde c}}\lambda_N^{-\sfrac{(d+2\alpha+2\beta)}{4}}}\,.
\end{equation*}
\end{lemma}

\begin{proof}[Proof of Theorem~\ref{t:disWeakCts}]
Given Lemma~\ref{l:ctsWeakH1small}, the proof of Theorem~\ref{t:disWeakCts} is identical to that of Theorem~\ref{t:disStrongCts}.
\end{proof}

As before, the proof of Corollary~\ref{c:ctsweak} only involves computing $H_4$ explicitly when the mixing rate function decays polynomially.
\begin{proof}[Proof of Corollary~\ref{c:ctsweak}]
  When the mixing rate function~$h$ is given by the power law~\eqref{e:powerh}, we compute $h^{-1}(t)=(c/t)^{1/p}$.
  By the definition of~$H_4$ (equation~\eqref{e:H4}), we have
\begin{multline*}
\exp\paren[\Big]{2^{(2p+1)/p}\norm{\nabla u}_{L^\infty}(c\sqrt{\tilde c})^{1/p}H_4(\nu)^{\frac{2\alpha+2\beta+d}{4p}}}\\
=
\frac{\norm{\nabla u}_{L^\infty}^2(2c\sqrt {\tilde c})^{1/p}}{2\nu} H_4(\nu)^{\frac{2\alpha+2\beta+d-4p}{4p}}\,.
\end{multline*}
Taking the logarithm shows
\begin{equation*}
  H_4(\nu) = O\paren[\Big]{H_4(\nu)\sim C\abs{\ln \nu}^{\frac{4p}{2\alpha+2\beta+d}}}
\end{equation*}
asymptotically as $\nu \to 0$.
Substituting this in~\eqref{e:taudStrongCtsWeak} yields~\eqref{e:ctsweak} as desired.
\end{proof}

\begin{proof}[Proof of Lemma~\ref{l:ctsWeakH1small}]
  Following the proof of Lemma~\ref{l:H1ByL2SmallCts}, we claim that~\eqref{e:lowerbound} still holds in our case, provided $\lambda_N$ and $t_0$ are chosen correctly.
  Indeed, note that~\eqref{e:tmpcts2} and~\eqref{e:tmpCts4} still hold, and the only difference here is that we need to bound the second term in~\eqref{e:tmpcts2} using the weak mixing assumption.
  Explicitly,~\eqref{e:hmixrateWeakCts} gives
  \begin{align}\label{e:tmpCtsWeak1}
    \nonumber
    \int_{\frac{t_0+s}{2}}^{t_0}\norm{P_N \phi_s(r)}^2\,dr &\leq \int_{\frac{t_0+s}{2}}^{t_0}\sum_{l=1}^N\abs{\ip{\phi_s(r),e_l}}^2\,dr 
    \\ 
    \nonumber
    &\leq \sum_{l=1}^N\frac{t_0-s}{2}h\Big(\frac{t_0-s}{2}\Big)^2\norm{\phi_s(0)}_{\alpha}^2\lambda_l^\beta
    \\
    \nonumber
    &
    \leq \frac{N(t_0-s)}{2}h \paren[\Big]{ \frac{t_0-s}{2}}^2\lambda_N^\beta\norm{\phi_{s,0}}_\alpha^2
    \\
    \nonumber
    &\leq \frac{N(t_0-s)}{2}h\Big(\frac{t_0-s}{2}\Big)^2\lambda_N^{\alpha+\beta}\norm{\theta_{s,0}}^2
    \\
    &\leq \frac{\tilde c (t_0-s)}{2}h\Big(\frac{t_0-s}{2}\Big)^2\lambda_N^{(d+2\alpha+2\beta)/2}\norm{\theta_{s,0}}^2\,.
  \end{align}
  Here the last inequality follows from the fact that our choice of~$\tilde c$ (in Remark~\ref{r:ctilde}) guarantees
  \begin{equation*}
    \frac{\tilde c \lambda_N^{d/2}}{2}
      \leq N
      \leq \tilde c \lambda_N^{d/2} \,,
  \end{equation*}
  for all $N$ sufficiently large.

Substituting~\eqref{e:tmpCts4} and~\eqref{e:tmpCtsWeak1} into~\eqref{e:tmpcts2}, we obtain
\begin{align*}
&\int_s^{t_0}\norm{\theta_s(r)}_1^2\,dr \\
&\quad 
\geq \frac{\lambda_N(t_0-s)\norm{\theta_{s,0}}^2}{4}
  \paren[\Big]{
    1
    - \tilde c\lambda_N^{(d+2\alpha+2\beta)/2}h\paren[\Big]{\frac{t_0-s}{2}}^2
    - \frac{\nu\lambda_N e^{2\norm{\nabla u}_{L^\infty}(t_0-s)}}{\norm{\grad u}_{L^\infty}^2 (t_0-s)}}\,.
\end{align*}
By our choice of $\lambda_N$ and $t_0$, we have
\begin{align*}
\tilde c\lambda_N^{(d+2\alpha+2\beta)/2}h\paren[\Big]{\frac{t_0-s}{2}}^2
  \leq \frac{1}{4}\,,
  \qquad\text{and}\qquad
\frac{\nu\lambda_N e^{2\norm{\nabla u}_{L^\infty}(t_0-s)}}{\norm{\grad u}_{L^\infty}^2 (t_0-s)}\leq \frac{1}{4}\,,
\end{align*}
from which equation~\eqref{e:lowerbound} follows.
This finishes the proof.
\end{proof}

\subsection{The Principal Eigenvalue with Dirichlet Boundary Conditions}\label{s:peval}

We now prove Proposition~\ref{p:eval} estimating the principal eigenvalue of $-\nu \lap + (u \cdot \grad)$ in a bounded domain with Dirichlet boundary conditions.

\begin{proof}[Proof of Proposition~\ref{p:eval}]
  For notational convenience we will write $\mu_0$ to denote $\mu_0(\nu , u)$.
  Let $\phi_0 = \phi_0(\nu, u)$ be the principal eigenfunction of the operator $-\nu \lap + (u \cdot \grad)$.
  Then we know
  \begin{equation*}
    \psi(x, t) \defeq \phi_0(x) e^{-\mu_0 t}
  \end{equation*}
  satisfies the advection diffusion equation
  \begin{equation*}
    \partial_t \psi + u \cdot \grad \psi - \nu \lap \psi = 0\,,
  \end{equation*}
  with initial data $\phi_0$.
  Consequently $\norm{\psi(t)} = e^{-\mu_0 t} \norm{\psi(0)}$.
  This forces $\tau_d \geq 1 / \mu_0$ proving~\eqref{e:lamda0taud} as claimed.
\end{proof}

\appendix
\section{Weak and Strong Mixing Rates}\label{s:mixrates}

In this appendix we provide a brief introduction to mixing and, in particular, analyze the notions of weak and strong weak mixing rates as in Definition~\ref{d:mixrate}.
Recall that $M$ is a $d$-dimensional Riemannian manifold with volume form normalized so that the total volume of $M$ is $1$.
A volume preserving diffeomorphism $\varphi \colon M \to M$ is said to be \emph{mixing} (or \emph{strongly mixing}) if for every pair of Borel sets $A, B \subseteq M$, we have
\begin{equation}\label{e:mixdef1}
  \lim_{n\to \infty} \vol(\varphi^{-n} (A) \cap B)
  = \vol(A) \vol(B) \,.
\end{equation}
Roughly speaking, this says that for every Borel set $A$, successive iterations of the map $\varphi$  will stretch and fold it over $M$ so that it eventually the fraction of \emph{every} fixed region $B \subseteq M$ occupied by~$A$ will approach $\vol(A)$.
For a comprehensive review of mixing we refer the reader to~\cites{KatokHasselblatt95,SturmanOttinoEA06}.

Approximating by simple functions we see that~\eqref{e:mixdef1} immediately implies that for any $f, g \in L^2_0$, we have%
\footnote{
  Recall $L^2_0$ is the set of all mean zero square integrable functions, and $U \colon L^2_0 \to L^2_0$ is the Koopman operator defined by $Uf = f \circ \varphi$.
}
\begin{equation*}
  \lim_{n \to \infty} \ip{U^n f, g} = 0\,.
\end{equation*}
Thus, one can quantify the \emph{mixing rate} by requiring the correlations $\ip{U^n f, g}$ to decay at a particular rate.
Since these are linear in $f, g$, a natural first attempt is to require
\begin{equation}\label{e:mixRateWrong}
  \abs[\big]{ \ip{U^n f, g} } \leq h(n) \norm{f} \, \norm{g}\,,
\end{equation}
for some decreasing sequence $h(n)$ that vanishes at infinity.
This, however, is impossible.
Indeed using duality, equation~\eqref{e:mixRateWrong} immediately implies
\begin{equation}\label{e:UnFto0}
  \norm{U^n f} \leq h(n) \xrightarrow{n \to \infty} 0\,.
\end{equation}
Of course, $U$ is a unitary operator and hence we must also have~$\norm{U^n f} = \norm{f}$, which is in direct contradiction to~\eqref{e:UnFto0}.

To circumvent this difficulty, one uses stronger norms of $f$ and $g$ on the right of~\eqref{e:mixRateWrong}.
The traditional choice in the dynamical systems literature is to use H\"older norms.
However, following Fannjiang et.\ al.~\cite{FannjiangWoowski03,FannjiangNonnenmacherEA04,FannjiangNonnenmacherEA06}, we use Sobolev norms instead, as it is more convenient for our purposes.
This is the content of the first part of Definition~\ref{d:mixrate}, and is repeated here for convenience.
\begin{definition}\label{d:mixrate1}
  Let $h \colon \N \to (0, \infty)$ be a decreasing function that vanishes at infinity, and~$\alpha, \beta > 0$.
  We say that $\varphi$ is \emph{strongly $\alpha$, $\beta$ mixing with rate function~$h$} if for all $f \in \dot H^\alpha$, $g \in \dot H^\beta$  the associated Koopman operator $U$ satisfies 
  \begin{equation}\label{e:hmixrateStrong2}
    \abs[\big]{\ip{U^n f, g}}
      \leq h\paren{n} \norm{f}_{\alpha}  \norm{g}_{\beta}\,.
  \end{equation}
\end{definition}

\begin{remark}\label{r:smix}
  We saw above that there are no strongly $\alpha$, $\beta$ mixing diffeomorphisms when both $\alpha = 0$ and $\beta = 0$.
  The same argument shows that there are no strongly $\alpha$, $\beta$ mixing diffeomorphisms when either $\alpha = 0$ and $\beta = 0$, as long as the rate function $h$ vanishes at~$\infty$.
  Thus, in Definition~\ref{d:mixrate1}, we need to ensure that both $\alpha$ and $\beta$ are strictly positive.
\end{remark}

\begin{remark}
If~$U$ is simply a unitary operator, then the rate function~$h$ can decay arbitrarily fast.
However, when~$U$ is the Koopman operator associated with a smooth map~$\varphi$, the rate function can decay at most exponentially.
To see this, note that for $k \in \N$ we have $\norm{U f}_k \leq c_k \norm{f}_k$ for some finite constant~$c_k = c_k( \norm{\varphi}_{C^k} ) > 1$.
Iterating this~$n$ times, choosing $k = \ceil{\beta}$, and $g = U^n f$ in~\eqref{e:hmixrateStrong2} gives
\begin{equation*}
  \norm{f}^2 = \norm{U^n f}^2 \leq h(n) \norm{f}_\alpha \norm{f}_{k} c_k^n\,,
\end{equation*}
forcing
\begin{equation*}
  h(n) \geq \frac{\norm{f}^2 c_k^{-n}}{\norm{f}_\alpha \norm{f}_k}\,.
\end{equation*}
\end{remark}

\begin{remark}
By duality equation~\eqref{e:hmixrateStrong2} implies that if~$\varphi$ is $\alpha$, $\beta$ mixing with rate function~$h$, then
\begin{equation}\label{e:HminusBeta}
  \norm{U^n f}_{-\beta} \leq h(n) \norm{f}_\alpha\,.
\end{equation}
In particular, this implies $\norm{U^n f}_{-\beta} \to 0$ as $n \to \infty$, and this has been used by many authors~\cites{MathewMezicEA05,LinThiffeaultEA11,Thiffeault12,IyerKiselevEA14} to quantify (strong) mixing.
\end{remark}

We now address the role of $\alpha$, $\beta$ in Definition~\ref{d:mixrate1}.
It turns out that if~$\varphi$ is strongly $\alpha$, $\beta$ mixing with rate function~$h$, then it must be strongly $\alpha'$, $\beta'$ mixing (at a particular rate) for \emph{every} $\alpha'$, $\beta' > 0$.
This is stated precisely in the following proposition.

\begin{proposition}
  Suppose for some $\alpha, \beta > 0$, the map~$\varphi$ is strongly~$\alpha$, $\beta$ mixing with rate function~$h$. 
  Then, for any $\alpha'$, $\beta' > 0$, the map~$\varphi$ is strongly~$\alpha'$, $\beta'$ mixing with rate function
  \begin{equation*}
    h'(t) \defeq \lambda_1^{-\gamma} h(t)^\delta\,,
  \end{equation*}
  where
  \begin{gather*}
    \gamma \defeq
     \frac{1}{2}\paren[\Big]{ (\alpha' - \alpha)^+
      + (\beta' - \beta)^+
      +  (\beta' \varmin \beta) \paren[\Big]{ 1 - \frac{\alpha'}{\alpha}}^+
      +  (\alpha' \varmin \alpha) \paren[\Big]{1 - \frac{\beta'}{\beta}}^+}\,,
    \\
    \llap{and\qquad}
    \delta \defeq \frac{(\alpha' \varmin \alpha) (\beta' \varmin \beta)}{\alpha \beta}\,.
  \end{gather*}
\end{proposition}
In particular, if for some~$\alpha$, $\beta > 0$, $\varphi$ is strongly $\alpha$, $\beta$ exponentially mixing, then it is strongly $\alpha'$, $\beta'$ exponentially mixing for all $\alpha'$, $\beta' > 0$.

\begin{proof}
  If $\beta \leq \beta'$, then we note
  \begin{align*}
    \norm{U^n f}_{-\beta'}
      &\leq
	\lambda_1^{(\beta - \beta')/2}
	\norm{U^n f}_{-\beta}
    \\
      &\leq
	\lambda_1^{(\beta - \beta')/2}
	h(n)
	\norm{f}_{\alpha}\,.
  \end{align*}
  On the other hand, if $\beta > \beta'$ then by Sobolev interpolation we have
  \begin{align*}
    \norm{U^n f}_{-\beta'}
      &\leq
	\norm{U^n f}_{-\beta}^{\beta' / \beta}
	\norm{U^n f}^{1 - \beta' / \beta}
    \\
      &\leq
	h(n)^{\beta' / \beta}
	\norm{f}_{\alpha}^{\beta' / \beta}
	\norm{f}^{1 - \beta' / \beta}
    \\
      &\leq
	\lambda_1^{-\alpha(1 - \beta' / \beta)/2}
	h(n)^{\beta' / \beta}
	\norm{f}_{\alpha}\,.
  \end{align*}
  This shows that~$\varphi$ is strongly $\alpha$, $\beta'$ mixing with rate function
  \begin{equation*}
    h_1(t)
      \defeq
	\lambda_1^{-(\beta' - \beta)^+ /2- \alpha (1 - \beta' / \beta)^+/2}
	h(t)^{(\beta' / \beta) \varmin 1}\,.
  \end{equation*}

  By dualizing, we see $\varphi^{-1}$ is strongly $\beta'$, $\alpha$ mixing with rate function~$h_1$.
  Thus, using the above argument, $\varphi^{-1}$ must be $\beta'$, $\alpha'$ mixing with rate function
  \begin{align*}
    h'(t)
      &\defeq
	\lambda_1^{-(\alpha' - \alpha)^+/2 - \beta' (1 - \alpha' / \alpha)^+/2}
	h_1(t)^{(\alpha' / \alpha) \varmin 1}
    \\
      &= \lambda_1^{-\gamma} h(t)^\delta\,,
  \end{align*}
  as desired.
\end{proof}
\medskip

We now turn our attention to weak mixing.
Recall that the dynamical system generated by~$\varphi$ is said to be \emph{weakly mixing} if for every pair of Borel sets $A, B \subseteq M$, we have
\begin{equation}\label{e:wmixdef1}
  \lim_{n\to \infty} \frac{1}{n} \sum_{k = 0}^{n-1} 
    \abs[\big]{ \vol(\varphi^{-k} (A) \cap B)
	- \vol(A) \vol(B) } = 0\,.
\end{equation}
Clearly strongly mixing implies weakly mixing, but the converse is false (see for instance~\cite{AnosovKatok70}).
Approximating by simple functions, and using the fact that~$U$ is $L^2$ bounded, one can show that~\eqref{e:wmixdef1} holds if and only if
\begin{equation}\label{e:wmixdef2}
  \lim_{n\to \infty} \frac{1}{n} \sum_{k = 0}^{n-1} 
    \abs[\big]{ \ip{U^n f, g} }^2 = 0\,,
\end{equation}
for all $f, g \in L^2_0$ (see for instance~\cite[Theorem~9.19 (iv)]{EisnerFarkasEA15}).
We can now quantify the \emph{weak mixing rate} by by imposing a rate of convergence in~\eqref{e:wmixdef2}.
This is the content of the second part of Definition~\ref{d:mixrate}, and is repeated here for convenience.
\begin{definition}\label{d:wmixrate}
  Let $h \colon \N \to (0, \infty)$ be a decreasing function that vanishes at infinity.
  Given $\alpha, \beta \geq 0$, we say that $\varphi$ is \emph{weakly $\alpha$, $\beta$ mixing with rate function~$h$} if for all $f \in \dot H^\alpha$, $g\in \dot H^\beta$ and $n \in \N$ the associated Koopman operator $U$ satisfies 
  \begin{equation}\label{e:hmixrateWeakAppendix}
    \paren[\Big]{ \frac{1}{n} \sum_{k = 0}^{n-1}
      \abs[\big]{\ip{U^kf,g}}^2 }^{1/2}
      \leq h\paren{n} \norm{f}_{\alpha}  \norm{g}_{\beta}\,.
  \end{equation}
\end{definition}

As mentioned in Remark~\ref{r:smix}, when defining strong mixing rates, we need to consider stronger norms of both the test functions $f$ and $g$ (i.e.\ we needed both $\alpha > 0$ and $\beta > 0$).
For weak mixing rates, however, one need not use stronger norms of both \emph{both} the test functions $f$ and $g$.
Indeed Proposition~\ref{p:catmapWeakMix} shows that for toral automorphisms, either~$\alpha$ or~$\beta$ (but not both) may be chosen to be~$0$.
We now show that it is impossible to choose \emph{both} $\alpha = 0$ and $\beta = 0$, and thus~\eqref{e:hmixrateWeakAppendix} must involve a stronger norm of either~$f$, or of $g$.

\begin{proposition}
  Let~$h$ be any function that decreases to~$0$.
  Then there does not exist any diffeomorphism~$\varphi$ which is  weakly $0$, $0$ mixing with rate function~$h$.
\end{proposition}
\begin{proof}
  Suppose for contradiction there exists a diffeomorphism~$\varphi$ which is weakly $0$, $0$ mixing with some rate function~$h$.
  Recall, by definition, the rate function~$h$ must vanish at infinity.
  We will show that for any fixed $N \in \N$,
  \begin{equation}\label{e:wmix1}
    \sup_{\norm{f} = \norm{g} = 1}
      \paren[\Big]{
	\frac{1}{N} \sum_{k = 0}^{N - 1} \abs{\ip{U^k f, g}}^2
      }
      \geq \frac{1}{2} \,.
  \end{equation}
  This immediately implies~$h(N) \geq 1/2$, contradicting the fact that $h$ vanishes at~$\infty$.

  Thus to finish the proof we only need to prove~\eqref{e:wmixdef1}.
  For this, note that~$\varphi$ must be weakly mixing (as $h$ vanishes at infinity).
  Since weakly mixing maps are ergodic, we know (see for instance~\cite{Walters82}) that almost every point has a dense orbit.
  Let $x_0$ be one such point, and note that $\varphi^n(x_0) \neq x_0$ for all $n \neq 0$.
  By continuity of~$\varphi$ we can now find a~$\delta = \delta(N) > 0$ such that 
  \begin{align*}
    B(x_0, \delta)\cap \varphi^k\paren[\big]{B(x_0, \delta)}=\emptyset\,,
    \quad\text{whenever }
      0 < \abs{k} < 2N \,.
  \end{align*}

  Now let~$\rho \in C_c(B(x_0, \delta)) \cap L^2_0(M)$ be such that $\norm{\rho} = 1$, and define the test functions $f, g$ by
  \begin{align*}
    f=\frac{1}{\sqrt{N}}\sum_{i=0}^{N-1}U^{-i} \rho\,,
    \qquad\text{and}\qquad
    g=\frac{1}{\sqrt{N}}\sum_{i=0}^{N-1}U^{i} \rho\,.
  \end{align*}
  Note by definition of~$\rho$ we have $\ip{U^i \rho, U^j\rho} = 0$ whenever $0 < \abs{i - j} < 2N$.
  This implies $\norm{f} = \norm{g} = 1$, and
  \begin{align*}
    \frac{1}{N}\abs[\Big]{\sum_{k=0}^{N-1} \ip{U^k f, g}}
      =\frac{1}{N^2}\sum_{i,j,k=0}^{N-1} \ip{U^{k-i} \rho, U^j \rho}
      =\frac{1}{N^2}\sum_{k=0}^{N-1} \sum_{i=0}^k 1
      =\frac{N + 1}{2N}
      \geq \frac{1}{2}\,.
  \end{align*}
  This proves~\eqref{e:wmix1} as desired, finishing the proof.
\end{proof}

\section{A characterization of relaxation enhancing maps on the torus}\label{s:ckrz}

We devote this appendix to proving Proposition~\ref{p:re} characterizing maps~$\varphi$ for which $\nu \tau_d \to 0$.
The main idea behind the proof is the same as that used in~\cite{ConstantinKiselevEA08,KiselevShterenbergEA08}.
The backward implication is simpler, and we present the proof of it first.

\begin{proof}[Proof of the backward implication in Proposition~\ref{p:re}]
  For the backward implication, we need to assume~$\nu \tau_d \to 0$, and show that the associated Koopman operator $U$ has no non-constant eigenfunctions in $\dot H^1$.
  Suppose, for sake of contradiction, that $f \in \dot H^1$ is an eigenfunction, normalized so that $\norm{f}=1$, and let~$\lambda$ be the corresponding eigenvalue. 
  Choosing~$\theta_0 = f$, and defining~$\theta_n$ by~\eqref{e:pulseddiff2} we observe
  \begin{align*}
    \MoveEqLeft
    \abs{\ip{\theta_{n+1},f}-\ip{U\theta_n, f}} =\abs[\Big]{\sum_k (1-e^{-\nu\lambda_k})(U\theta_n)^\wedge(k) \overline{\hat f(k)}}\\
    &\leq \nu \paren[\Big]{\sum_k \frac{1-e^{-\nu\lambda_k}}{\nu}\abs{(U\theta_n)^\wedge(k)}^2}^{\sfrac{1}{2}} \paren[\Big]{\sum_k\frac{1-e^{-\nu\lambda_k}}{\nu}\abs{\hat{f}(k)}^2}^{\sfrac{1}{2}}\\
    &\leq \nu \paren[\Big]{\sum_k \frac{1-e^{-\nu\lambda_k}}{\nu}\abs{(U\theta_n)^\wedge(k)}^2}^{\sfrac{1}{2}} \paren[\Big]{\sum_k\frac{1-e^{-\nu\lambda_k}}{\nu}\abs{\hat{f}(k)}^2}^{\sfrac{1}{2}}\\
    &\leq  \nu (\Enu \theta_n)^{\sfrac{1}{2}} \norm{ f}_{1}
    \leq \frac{\nu}{2}\Enu\theta_n+\frac{\nu}{2}\norm{f}_1^2\,.
  \end{align*}
Using equation \eqref{e:l2decay1}, this gives
\begin{align*}
\abs{\ip{\theta_{n+1},f}-\ip{U\theta_n, f}}\leq \frac{1}{2}(\norm{\theta_n}^2-\norm{\theta_{n+1}}^2 )+ \frac{\nu}{2}\norm{f}_{1}^2\,,
\end{align*}
which implies 
\begin{align*}
\abs{\ip{\theta_{n+1},f}}-\abs{\ip{U\theta_n, f}} \geq -\frac{1}{2}(\norm{\theta_n}^2-\norm{\theta_{n+1}}^2)- \frac{\nu}{2}\norm{f}_{1}^2\,.
\end{align*}

Since $\ip{U\theta_n, f}=\ip{\theta_n, U^*f}=\lambda \ip{ \theta_n, f}$, and $\abs{\lambda}=1$, the above implies
\[
\abs{\ip{\theta_{n+1},f}}-\abs{\ip{\theta_n, f}} \geq -\frac{1}{2}(\norm{\theta_n}^2-\norm{\theta_{n+1}}^2)- \frac{\nu}{2} \norm{f}_{1}^2\,.
\]
Iterating this gives
\begin{align*}
  \abs{\ip{\theta_{n},f}}-\abs{\ip{f,f}}
    &\geq 
	 -\frac{1}{2}(\norm{f}^2-\norm{\theta_{n}}^2)- \frac{n\nu}{2} \norm{f}_{1}^2\,,
\end{align*}
since $\theta_0=f$.
Thus
\begin{align*}
  \abs{\ip{\theta_{n},f}}
    \geq \frac{1}{2}\norm{f}^2  +\frac{1}{2}\norm{\theta_n}^2- \frac{n\nu}{2} \norm{f}_{1}^2\geq \frac{1}{2}- \frac{n\nu}{2} \norm{f}_{1}^2\,.
\end{align*}

Now choosing $n$ to be the dissipation time $\tau_d$ gives
\begin{align*}
 \frac{1}{e} \geq  \abs{\ip{\theta_{\tau_d},f}}
    \geq \frac{1}{2}- \frac{\tau_d\nu}{2}\norm{f}_{1}^2\,,
\end{align*}
and hence
\begin{equation*}
  \nu\tau_d\geq \frac{e-2}{e\norm {f}_1^2}\,.
\end{equation*}
This contradicts the assumption $\nu\tau_d \to 0$ as $\nu \to 0$, finishing the prof.
\end{proof}

For the other direction, we need two lemmas.
The first is an application of the discrete RAGE theorem.
\begin{lemma}\label{a:1}
  Let $K\subset S=\set{\phi \in L^2_0 \st \norm{\phi}=1}$ be a compact set.
  Let $P_c$ be the spectral projection on the continuous spectral subspace in the spectral decomposition of~$U$.
  For any $N,\delta >0$, there exists $n_c(N,\delta,K)$ such that for all $n\geq n_c$ and any $\phi\in K$, we have
\begin{align}\label{e:rage1}
\frac{1}{n-1}\sum_{i=1}^{n-1} \norm{P_N U^i P_c\phi}^2\leq \delta\,.
\end{align} 
\end{lemma}
\begin{proof}
  \GI[2018-06-11]{TODO: \cite{CyconFroeseEA87} has the continuous time RAGE. Eventually find a reference for the discrete time version.}
  Define
  \begin{align*}
    f(n,\phi) \defeq \frac{1}{n-1}\sum_{i=1}^{n-1} \norm{P_N U^i P_c\phi}^2\,.
  \end{align*}
  Recall that by the RAGE theorem~\cite{CyconFroeseEA87} we have
  \begin{align*}\label{e:rage}
    \lim_{n\to \infty} \frac{1}{n}\sum_{i=0}^{n-1} \norm{A U^i P_c\phi}^2=0\,,
    \quad\text{for any compact operator }A \,,
  \end{align*}
  and hence for all~$\phi$, $f(n, \phi) \to 0$ as $n \to \infty$.
  Thus, to finish the proof, we only need to show that this convergence is uniform on compact sets.

  To prove this, it is enough to prove the functions $f(n,\cdot)$ are equicontinuous.
  For this observe that for any $\phi_1, \phi_2\in S$ we have
  \begin{align*}
    \MoveEqLeft
    |f(n,\phi_1)-f(n,\phi_2)|\\
    &\leq \frac{1}{n-1}\sum_{i=1}^{n-1}\abs[\big]{\norm{P_NU^iP_c\phi_1}-\norm{P_NU^iP_c\phi_2}}\paren[\big]{\norm{P_NU^iP_c\phi_1}+\norm{P_NU^i P_c\phi_2}}\\
    &\leq \frac{1}{n-1}\sum_{i=1}^{n-1} \norm{\phi_1-\phi_2}\paren[\big]{\norm{\phi_1}+\norm{\phi_2}}\\
    &\leq 2\norm{\phi_1-\phi_2}\,.
  \end{align*}
  This shows equicontinuity, finishing the proof.
\end{proof}

\begin{lemma}\label{a:2}
Assume that the Koopman operator $U$ has no eigenfunctions in $\dot H^1$.  Let $P_p$ be the spectral projection on its point spectral subspace. Let $K$ be a compact subset of $S$. Define the set $K_1=\{\phi \in K\,|\, \norm{P_p\phi}\geq \frac{1}{2}\}$. Then for any $C>0$, there exist $N_p(C,K)$ and $n_p(C,K)$ such that for any $N\geq N_p(C,K)$, any $n\geq n_p(C,K)$, and any $\phi \in K_1$, 
\begin{align}
\frac{1}{n-1}\sum_{i=1}^{n-1}\norm{P_NU^iP_p\phi}_1^2 \geq C\,.
\end{align}
\end{lemma}

The proof of this is the same as Lemma 3.3 in \cite{ConstantinKiselevEA08} and we do not present it here.
We can now finish the proof of Proposition~\ref{p:re}. 
\begin{proof}[Proof of the forward implication in Proposition~\ref{p:re}]
  For this direction we are given that $U$ has no eigenfunctions in $\dot H^1$, and need to show $\nu \tau_d \to 0$ as $\nu \to 0$.
  We will show that for any $\eta>0$, 
  \begin{align}\label{e:b3}
    \norm[\Big]{\theta\paren[\Big]{ \ceil[\Big]{\frac{\eta}{\nu}}} }
      \to 0 \text{ as } \nu\to 0\,,
  \end{align}
  which immediately implies $\nu\tau_d \to 0$ as~$\nu \to 0$.

  To prove~\eqref{e:b3}, we need to show for any given $\eta, \epsilon$, there exists $\nu_0$, such that for any $\nu\leq \nu_0$, we have $\norm{\theta(\ceil{\frac{\eta}{\nu}})}^2\leq \epsilon$ for any initial $\theta_0\in H$ with $\norm{\theta_0}=1$.
  We choose $N$ large enough satisfying $e^{-\lambda_N\eta/80}\leq \epsilon$.  Denote $K=\{\phi\in S\,|\, \norm{\phi}^2\leq \lambda_N\}$, and $K_1=\{ \phi \in K\,|\, \norm{P_p\phi}\geq \frac{1}{2}\}$. Let $n_1$ be 
  \begin{align*}
    n_1=\max\set[\Big]{2, n_p(5\lambda_N,K), n_c\paren[\Big]{ N,\frac{1}{20},K} }\,,
  \end{align*}
  and choose $\nu_0$ small enough so that
  \begin{equation*}
    n_1\leq \frac{\eta}{2\nu_0}\,,
    \qquad
    \nu_0n_1^2\leq \frac{1}{\lambda_N}
    \qquad\text{and}\qquad
    \frac{n_1^2\nu_0\lambda_N\norm{\grad \varphi}_{L^\infty}^{2n_1+2}}{(n_1-1)(\norm{\grad \varphi}_{L^\infty}^2-1)}\leq \frac{1}{4}\,.
  \end{equation*}

  Note that if $\Enu\theta_n \geq \lambda_N\norm{\theta_n}^2$ for all $n\in [0,\ceil{\sfrac{\eta}{\nu}}]$, then we have
  \begin{equation*}
    \norm[\Big]{\theta \paren[\Big]{ \ceil[\Big]{\frac{\eta}{\nu}}}}^2 \leq e^{-\nu \lambda_N \ceil{\sfrac{\eta}{\nu}}} \leq e^{-\lambda_N\eta} \leq \epsilon\,.
  \end{equation*}
  If not, let $n_0\in [0,\ceil{\sfrac{\eta}{\nu}}]$ be the first time satisfying $\Enu \theta_{n_0}< \lambda_N \norm{\theta_{n_0}}^2$.
  Similar to \eqref{H1small} we have $\norm{\theta_{n_0+1}}_1^2< \lambda_N\norm{\theta_{n_0+1}}^2$.
  We  claim that our choice of~$n_1$ will guarantee
  \begin{equation}\label{e:b4}
    \norm{\theta_{n_0+n_1}}^2 \leq e^{-\lambda_N \nu n_1/40}\norm{\theta_{n_0}}^2\,.
  \end{equation}
  Given~\eqref{e:b4}, we can find $\tilde n\in [\eta/(2\nu),\eta/\nu]$ such that $\norm{\theta(\ceil{\eta/\nu})}^2 \leq \norm{\theta_{\tilde n}}^2 \leq e^{-\lambda_N \nu \tilde n/40} \leq e^{-\lambda_N\eta/80}\leq \epsilon$, proving~\eqref{e:b3} as desired.

  Thus it only remains to prove~\eqref{e:b4}.
  For this, define $\phi_m= U^{m-1}\theta_{n_0+1}$, and observe 
  \begin{equation*}
    \frac{\phi_1}{\norm{\phi_1}}
      = \frac{\theta_{n_0+1}}{\norm{\theta_{n_0+1}}} \in K\,,
      \quad
      P_c \phi_m=U^{m-1}P_c\theta_{n_0+1}\,,
      \quad\text{and}\quad
      P_p \phi_m=U^{m-1}P_p\theta_{n_0+1}\,.
  \end{equation*}
  We now consider two cases.

  \begin{proofcases}
    \case[$\norm{P_c\theta_{n_0+1}}^2\geq \frac{3}{4}\norm{\theta_{n_0+1}}^2$
      (or equivalently $\norm{P_p\theta_{n_0+1}}^2\leq \frac{1}{4}\norm{\theta_{n_0+1}}^2$)]
  In this case, we have 
  \begin{align}\label{e:atmp1}
    \MoveEqLeft
    \nonumber
    \sum_{m=1}^{n_1-1}\Enu\theta_{n_0+m} \geq 2\sum_{m=1}^{n_1-1}\norm{\theta_{n_0+1+m}}_1^2\\
    \nonumber
    &\geq 2\lambda_N \sum_{m=1}^{n_1-1}\norm{(I-P_N)\theta_{n_0+1+m}}^2\\
    &\geq \lambda_N\sum_{m=1}^{n_1-1}\norm{(I-P_N)\phi_{m+1}}^2-2\lambda_N\sum_{m=1}^{n_1-1}\norm{(I-P_N)(\theta_{n_0+1+m}-\phi_{m+1})}\,.
  \end{align}
  By direct calculation, we also have
  \begin{align*}
    \norm{(I-P_N)\phi_{m+1}}^2 
    & \geq \frac{1}{2}\norm{(I-P_N)P_c\phi_{m+1}}^2-\norm{(I-P_N)P_p\phi_{m+1}}^2\\
    &\geq \frac{1}{2}\norm{U^mP_c\theta_{n_0+1}}^2-\frac{1}{2}\norm{P_NU^mP_c\theta_{n_0+1}}^2-\norm{U^mP_p\theta_{n_0+1}}^2\\
    &= \frac{1}{2}\norm{P_c\theta_{n_0+1}}^2-\frac{1}{2}\norm{P_NU^mP_c\theta_{n_0+1}}^2-\norm{P_p\theta_{n_0+1}}^2\,.
  \end{align*}
  By Lemmas~\ref{a:1},\ref{a:2}, and the choice of $n_1$, we have
  \begin{align}\label{e:atmp2}
    \frac{1}{n_1-1}\sum_{m=1}^{n_1-1}\norm{(I-P_N)\phi_{m+1}}^2 \geq \frac{1}{10}\norm{\theta_{n_0+1}}^2\,.
  \end{align}
  Substituting~\eqref{energysub1} and \eqref{e:atmp2} in~\eqref{e:atmp1} gives
  \begin{align*}
    \sum_{m=1}^{n_1-1}\Enu \theta_{n_0+m}\geq \frac{\lambda_N(n_1-1)}{20}\norm{\theta_{n_0+1}}^2\,.
  \end{align*}
  Since $\norm{\theta_{n_0+n_1}}^2 =\norm{\theta_{n_0+1}}^2 -\nu \sum_{m=1}^{n_1-1}\Enu \theta_{n_0+m}$, we further have
  \begin{align*}
    \norm{\theta_{n_0+n_1}}^2 &\leq \Big(1-\frac{\nu \lambda_N(n_1-1)}{20}\Big)\norm{\theta_{n_0+1}}^2 \\
    & \leq \paren[\Big]{1-\frac{\nu\lambda_N n_1}{40}}\norm{\theta_{n_0}}^2 \leq e^{-\frac{\nu \lambda_N n_1}{40}}\norm{\theta_{n_0}}^2\,.
  \end{align*}

  \case[%
    $\norm{P_p\theta_{n_0+1}}^2\geq \frac{1}{4}\norm{\theta_{n_0+1}}^2$
    (or equivalently $\norm{P_c\theta_{n_0+1}}^2\leq \frac{3}{4}\norm{\theta_{n_0+1}}^2$)]
  By Lemma~\ref{a:2}, we have 
  \begin{align}\label{e:atmp3}
    \frac{1}{n_1-1}\sum_{m=1}^{n_1-1}\norm{P_N U^m P_p \theta_{n_0+1}}_1^2 \geq 5\lambda_N\norm{\theta_{n_0+1}}^2\,,
  \end{align}
  and Lemma~\ref{a:1} yields
  \begin{align}\label{e:atmp4}
    \frac{1}{n_1-1}\sum_{m=1}^{n_1-1}\norm{P_N U^m P_c \theta_{n_0+1}}_1^2 \leq \frac{\lambda_N}{20}\norm{\theta_{n_0+1}}^2\,.
  \end{align} 
  Combining~\eqref{e:atmp3} and~\eqref{e:atmp4}, we get
  \begin{align}\label{e:atmp5}
    \frac{1}{n_1-1}\sum_{m=1}^{n_1-1}\norm{P_N U^m  \theta_{n_0+1}}_1^2 \geq 2\lambda_N\norm{\theta_{n_0+1}}^2\,.
  \end{align}
  By~\eqref{energysub1} and~\eqref{e:epsilonH1Relation}, we have
  \begin{align*}
    \frac{1}{n_1-1}\sum_{m=1}^{n_1-1}\norm{\theta_{n_0+1+m}-\phi_{m+1}}^2 &\leq\frac{ n_1^2 \nu}{n_1-1} \sum_{m=1}^{n_1-1}\norm{U\theta_{n_0+1+m}}_1^2\\
    & \leq \frac{n_1^2\nu}{n_1-1} \sum_{m=1}^{n_1-1}\norm{\grad \varphi}_{L^\infty} ^{2m+2} \norm{\theta_{n_0+1}}_1^2\\
    & \leq \frac{n_1^2\nu\norm{\grad \varphi}_{L^\infty}^{2n_1+2}}{(n_1-1)(\norm{\grad \varphi}_{L^\infty}^2-1)}\norm{\theta_{n_0+1}}_1^2\\
    &\leq \frac{1}{4}\norm{\theta_{n_0+1}}^2\,,
  \end{align*}
  which implies 
  \begin{align}\label{e:atmp6}
    \frac{1}{n_1-1}\sum_{m=1}^{n_1-1}\norm{P_N(\theta_{n_0+1+m}-\phi_{m+1})}_1^2\leq \frac{\lambda_N}{4}\norm{\theta_{n_0+1}}^2\,.
  \end{align}
  Equation~\eqref{e:atmp5} together with ~\eqref{e:atmp6} gives 
  \begin{align}
    \sum_{m=1}^{n_1-1}\norm{\theta_{n_0+1+m}}_1^2\geq \sum_{m=1}^{n_1-1}\norm{P_N\theta_{n_0+1+m}}_1^2\geq \frac{\lambda_N}{2}(n_1-1)\norm{\theta_{n_0+1}}^2\,.
  \end{align}
  We now use~\eqref{e:epsilonH1Relation} again to get
  \begin{align*}
    \sum_{m=1}^{n_1-1}\Enu \theta_{n_0+m} \geq \lambda_N(n_1-1)\norm{\theta_{n_0+1}}^2\,,
  \end{align*}
  which, as before, yields
  \begin{align*}
    \norm{\theta_{n_0+n_1}}^2\leq e^{-\frac{\nu \lambda_N n_1}{2}}\norm{\theta_{n_0}}^2\,.
  \end{align*}
  This proves~\eqref{e:b4} as desired, finishing the proof.
  \end{proofcases}
\end{proof}

\bibliographystyle{halpha-abbrv}
\bibliography{refs,preprints}
\end{document}